\documentclass[12pt,reqno]{amsart}
\usepackage[usenames]{color}
\usepackage{amsmath,verbatim,graphicx,epstopdf,enumerate}
\usepackage[usenames]{color}
\usepackage{amsmath,verbatim,graphicx,enumerate,fancybox}
\usepackage[notref,notcite]{}
\usepackage[pagewise]{lineno}
\newtheorem{theorem}{Theorem}[section]
\newtheorem{lemma}[theorem]{Lemma}
\newtheorem{proposition}[theorem]{Proposition}
\newtheorem{corollary}[theorem]{Corollary}

\newtheorem{remark}[theorem]{Remark}

\numberwithin{equation}{section}

\usepackage{hyperref}

\newcommand{\PD}{\partial}

\newcommand{\I}{\mathrm{i}}

\newcommand{\la}{\mathcal{L}}
\newcommand{\Rn}{\mathbb{R}^n}

\newcommand{\bp}{\begin{prob}}
	\newcommand{\ep}{\end{prob}}
\newcommand{\bpr}{\begin{proof}}
	\newcommand{\epr}{\end{proof}}
\newcommand{\vp}{\varphi}

\newcommand{\lb}{\left(}

\newcommand{\ve}{\varepsilon}

\newcommand{\rb}{\right)}
\renewcommand{\d}{\mathrm{d}}
\newcommand{\wt}{\widetilde}

\newcommand{\Lc}{\mathcal{L}}

\newcommand{\Rb}{\mathbb{R}}
\newcommand{\Sb}{\mathbb{S}}

\newcommand{\om}{\omega}

\newcommand{\D}{\mathrm{d}}
\newcommand{\lt}{L^2\lb0,T;L^2(\Rn)\rb}
\newcommand{\lth}{L^2\lb 0,T; H^{-1}_{\lambda}(\Rn)\rb}
\renewcommand{\O}{\Omega}
\usepackage{amsmath}
\theoremstyle{definition}

\title[Partial data inverse problem]{A partial data inverse problem for the Convection-diffusion equation}

\author[Sahoo and Vashisth]{Suman Kumar Sahoo$^{\dagger}$ and Manmohan Vashisth$^{\ddagger}$}
\address{$^{\dagger}$ TIFR Centre for Applicable Mathematics, Bangalore 560065, India. 
	\newline\indent E-mail:{\tt \ suman@tifrbng.res.in}}
\address{$^{\ddagger}$ Beijing Computational Science Research Center, Beijing 100193, China.
	\newline
	\indent E-mail:{\tt\  mvashisth@csrc.ac.cn, manmohanvashisth@gmail.com}}
\begin{document}
		\maketitle
\begin{abstract}
	In this article we study the inverse problem of determining the  convection term and the time-dependent density coefficient appearing in the convection-diffusion equation. We prove the unique determination of these coefficients from the knowledge of solution measured on a subset of the boundary.
\end{abstract}


\section{Introduction}\label{Introduction}
Let $\Omega\subset\mathbb{R}^{n}$ with $n\geq 2$, be a bounded simply connected open set with $C^{2}$ boundary. For $T>0$, let $Q:=(0,T)\times\Omega$ and denote its lateral boundary by $\Sigma:=(0,T)\times\partial\Omega$. We consider the following initial boundary value problem
\begin{align}\label{definition of operator}
	\begin{aligned}
		\begin{cases}
			&\lb\partial_{t}-\sum_{j=1}^{n}\left(\partial_{j}+{A_{j}(t,x)}\right)^{2}+q(t,x)\rb u(t,x)=0,\ 
			(t,x)\in Q\\
			&u(0,x)=0,\ x\in\O\\
			&u(t,x)=f(t,x), \ (t,x)\in\Sigma.
		\end{cases}
	\end{aligned}
\end{align}
Throughout this article, we assume that $A_{j}\in W^{1,\infty}(Q)$ for $1\leq j\leq n$ and {$q\in L^{\infty}(Q)$}. Let us denote by 
\[{A(t,x):=\lb A_{1}(t,x),A_{2}(t,x),\cdots, A_{n}(t,x)\rb }\] and by 
\begin{align*}
	\Lc_{A,q}:=\partial_{t}-\sum_{j=1}^{n}\lb\partial_{j}+A_{j}(t,x)\rb^{2}+q(t,x).
\end{align*}
Before going to the main context of the article, let us briefly mention about the well-posedness of the forward problem. Following \cite{Caro_Kian_Convection_nonlinear}, define the spaces $\mathcal{K}_{0}$ and $\mathcal{H}_{T}$ by 
\begin{align*}
	\begin{aligned}
		&\mathcal{K}_{0}:= \left\{f|_{\Sigma}: \ f\in L^{2}\lb 0,T;H^{1}(\Omega)\rb\cap H^{1}\lb 0,T;H^{-1}(\Omega)\rb\  \mbox{and}\ f(0,x)=0,\ \mbox{for}\  x\in \Omega\right\} \\
		&\ \mbox{and}\ \mathcal{H}_{T}:=\left\{g|_{\Sigma}:\ g\in H^{1}\lb 0,T;H^{1}(\Omega)\rb\ \mbox{and}\ g(T,x)=0, \ \mbox{for}\ x\in\Omega\right\}.
	\end{aligned}
\end{align*}
As shown in \cite{Caro_Kian_Convection_nonlinear} (see also \cite{Lion-Magenes2}) that for $f\in \mathcal{K}_{0}$, Equation \eqref{definition of operator} admits a unique solution $u\in H^{1}\lb 0,T;H^{-1}(\Omega)\rb \cap L^{2}\lb 0,T;H^{1}(\Omega)\rb$ and  the operator $\mathcal{N}_{A,q}u$  given by 
\begin{align*}
	\langle \mathcal{N}_{A,q}u, w|_{\Sigma}\rangle :=\int\limits_{Q}\lb -u\PD_{t}\overline{w}+\nabla_{x}u\cdot\nabla_{x}\overline{w}+2u A\cdot\nabla_{x}\overline{w}+(\nabla_{x}\cdot A) u\overline{w}-\lvert A\rvert^{2}u\overline{w}+qu\overline{w}   \rb\D x\D t
\end{align*}
is well-defined for $w\in H^{1}(Q)$ such that $w(T,x)=0$, for $x\in\Omega$. Note that if $A,q$ and $f$ are smooth enough then $\mathcal{N}_{A,q}u$  is given by 
\begin{align*}
	\mathcal{N}_{A,q}u=\lb\PD_{\nu}u+2\lb\nu\cdot A\rb u \rb|_{\Sigma}
\end{align*}
where $\nu$ stands for the outward unit normal vector to $\PD\Omega$  and  $u$ solution to \eqref{definition of operator}. Motivated by this and \cite{Caro_Kian_Convection_nonlinear}, we define 
the Dirichlet to Neumann (DN) map $\Lambda_{A,q}:\mathcal{K}_{0}\rightarrow \mathcal{H}_{T}^{*}$ by 
\begin{align}\label{DN map}
	\Lambda_{A,q}(f):=\mathcal{N}_{A,q}u
\end{align}
where $\mathcal{H}_{T}^{*}$ denotes the dual of space $\mathcal{H}_{T}$  and  $u$ is solution to    \eqref{definition of operator} with Dirichlet boundary data equal to $f$. Then  from   (\cite{Caro_Kian_Convection_nonlinear}, see Section $2$), we have that DN map $\Lambda_{A,q}$ defined by \eqref{DN map} is continuous from  $\mathcal{K}_{0}$ to $\mathcal{H}_{T}^{*}$.

In the present article we first consider the problem of unique recovery of  coefficients $A(x)$ and $q(t,x)$ appearing in \eqref{definition of operator} from the information of DN map  $\Lambda_{A,q}$ measured on a subset of $\Sigma$.  It is well-known [see \cite{Sun_magnetic}] that  one cannot determine coefficient $A(x)$ uniquely from DN map $\Lambda_{A,q}$ measured on $\Sigma$ and this is because of the gauge invariance associated with $A(x)$. So one can only hope to recover $A(x)$ uniquely upto a potential term however the coefficient $q(t,x)$ can be determined uniquely (see Theorem \ref{Main Theorem} in \S \ref{Main result statement} for more details).  Later as a corollary of Theorem \ref{Main Theorem}, we consider the problem of determining time-dependent coefficients $A(t,x)$ and $q(t,x)$ appearing in \eqref{definition of operator} from the partial information of DN map $\Lambda_{A,q}$. Using some extra assumption on $A(t,x)$ and  Theorem \ref{Main Theorem}, we show that time-dependent coefficients $ A(t,x)$   and $q(t,x)$ can be determined uniquely from the knowledge of DN map $\Lambda_{A,q}$ measured on a part of $\Sigma$ (see Corollary \ref{Corollary} below in \S \ref{Main result statement} for more details).

The initial boundary value problem \eqref{definition of operator} is known as a convection-diffusion equation with constant diffusion. The coefficients $A$ and $q$ are called convection term and density coefficient respectively. The convection-diffusion equations appear in chemical engineering, heat transfer and probabilistic study of diffusion process etc.

Determination of the coefficients from boundary measurements appearing in parabolic partial differential equations have been studied by several authors. Isakov in \cite{Isakov_Completeness} considered the problem of determining time-independent coefficient $q$ for the case when $A=0$ in \eqref{definition of operator} from the DN map and he proved the uniqueness result by showing the density of product of solutions (inspired by the work of  \cite{Sylvester_Uhlmann_Calderon_problem_1987}) in some Lebesgue space.  Avdonin and Seidman in  \cite{Avdonin_Seidman_parabolic} studied the problem of determining time-independent density coefficient $q(x)$  appearing in  \eqref{definition of operator} by using the boundary control method pioneered by Belishev, Kurylev, Lassas and others see \cite{Avdonin and Belishev,Belishev_BC_Method_Recent_Progress,Katchalov_Kurylev_Lassas_Book_2001} and references therein.
In \cite{Choulli_Book} Choulli proved  the stability estimate analogous to the uniqueness problem considered in \cite{Isakov_Completeness}.  In \cite{Deng_Yu_Yang_First_order_parabolic_1995}  problem of determining the first order coefficients appearing in a parabolic equations in one dimension from the data measured at final time is studied.  Cheng and Yamamoto in \cite{Chen_Yamamoto_parabolic} proved the unique determination of convection term $A(x)$ (when $q=0$ in \eqref{definition of operator}) from a single boundary measurement in two dimension. 
Gaitan and Kian \cite{Gaitan_Kian_Cylinderical_Stability} using the global Carleman estimate used for hyperbolic equations [see \cite{Bukhgeuim_Klibanov_hyperbolic}] proved the stable determination of  time-dependent coefficient $q(t,x)$ in a bounded waveguide.  Choulli and Kian in \cite{Choulli_Kian_parabolic_partial} proved the stability estimate for determining time-dependent coefficient $q$ from the partial DN map. For more works related to parabolic inverse problems, we refer to \cite{Bellassoued_Kian_Soccorsi_parabolic_single_measurement,Choulli_Abstract_IP,Choulli_Abstract_IP_Applications,Choulli_Kian_parabolic_Parabolic,Choulli_Kian_Soccorsi_waveguide,Choulli_Book,Deng_Yu_Yang_First_order_parabolic_1995,Gaitan_Kian_Cylinderical_Stability,Isakov_semilinear,Isakov_Book,Nakamura_Sasayama_Parabolic} and the references therein. We also mention the work of  \cite{BP,Bellassoued_Ferreira_anisotropic_Schrodinger,Bellassoued_Kian_Soccorsi_Scrodinger_infinite,Ibtissem_Magnetic_Scrodinger_time,Eskin_electromagentic_time,Kian_Phan_Socorsi_Carleman_Infinite_Cylinder,Kian_Phan_Socorsi_Unbounded,Kian_Soccorsi_Holder_stability_Scrodinger} related to dynamical Schr\"odinger equation and the work of \cite{Ibtissem_wave_equation,Eskin_hyperbolic_time_1,Eskin_hyperbolic_time_2,Rakesh_Symes_Uniqueness_1988,Salazar,Stefanov_inverse_scattering} for hyperbolic inverse problems. We refer to  \cite{Cheng_Yamamoto_Global_Convection_2D_DN,Cheng_Yamamoto_Steady,Pohjola_steady_state_CD} for steady state convection-diffusion equation.
Recently Caro and Kian in \cite{Caro_Kian_Convection_nonlinear} established the unique determination of convection coefficient together with non-linearity term appearing in the equation from the knowledge of DN map measured on $\Sigma$.

Inspired by the work of \cite{Choulli_Kian_parabolic_partial}, we consider the problem of determining the full first order space derivative perturbation of heat operator from the partial DN map. We have proved our uniqueness result by using the geometric optics solutions constructed using a Carleman estimate in a Sobolev space of negative order and inverting the ray transform of a vector field which is known only in a very small neighbourhood of fixed direction {$\omega_0\in \Sb^{n-1}:=\left\{x\in\Rb^{n}: \lVert x\rVert=1\right\}$}. For elliptic and hyperbolic inverse problems these kind of techniques have been used by several authors. Related to our work, we refer to   \cite{Bukhgeim_Uhlmann_Calderon_problem_partial_Cauchy_data_2002,Ferriera_Kenig_Sjostrand_Uhlmann_magnetic} for the elliptic case and to \cite{Bellassoued_Jellali_Yamamoto_Lipschitz_stability_hyperbolic,Bellassoued_Jellali_Yamamoto_stability_hyperbolic,Hu_Kian_Wave_equation, Kian_ptential_uniqueness_wave,Kian_potential_stability_wave,Kian_damping,Kian-Oksanen,Krishnan_Vashisth_Relativistic} for the hyperbolic case.

The article is organized as follows. In \S  \ref{Main result statement} we give the statement of the main result. \S \ref{Boundary Carleman estimate} contains the boundary Carleman estimate. In \S \ref{Contruction of go solutions} we construct the geometric optics solutions using a Carleman estimate in a  Sobolev space of negative order. In \S \ref{Integral identity} we derive an integral identity and \S \ref{proof of the main theorem} contains the proof of main Theorem \ref{Main Theorem} and Corollary \ref{Corollary}.
\section{Statement of the main result}\label{Main result statement}
We begin this section by fixing some notation which will be used to state the main result of this article.   Following \cite{Bukhgeim_Uhlmann_Calderon_problem_partial_Cauchy_data_2002}  
fix an $\omega_{0}\in\mathbb{S}^{n-1}$ and define the $\omega_{0}$-shadowed and $\omega_{0}$-illuminated faces by 
\begin{align*}
	\partial\Omega_{+,\omega_{0}}:=\left\{x\in\partial\Omega:\ \nu(x)\cdot\omega_{0}\geq 0 \right\},\ \ \partial\Omega_{-,\omega_{0}}:=\left\{x\in\partial\Omega:\ \nu(x)\cdot\omega_{0}\leq 0 \right\}
\end{align*}
of $\partial\Omega$ where $\nu(x)$ is outward unit normal to $\partial\Omega$ at $x\in\partial\Omega$. Corresponding to $\partial\Omega_{\pm,\omega_{0}}$, we denote the lateral boundary parts by 
$\Sigma_{\pm,\omega_{0}}:=(0,T)\times\partial\Omega_{\pm,\omega_{0}}$. We denote by $F=(0,T)\times F'$ and $G=(0,T)\times G'$ where $F'$ and $G'$ are small enough open neighbourhoods of $\partial\Omega_{+,\omega_{0}}$ and  $\partial\Omega_{-,\omega_{0}}$ respectively in $\partial\Omega$. 

Since $\Omega$ is bounded and $T<\infty$, so we can choose a smallest $R>0$ such that $\overline{Q}\subset B(0,R)$ where $B(0,R)\subset\Rb^{1+n}$ is a ball of radius $R$ with center at origin.  Now we define admissible set $\mathcal{A}$ of vector fields $A(t,x)$ appearing in \eqref{definition of operator} by 
\begin{align}\label{Definition of admissible set}
	\mathcal{A}:=\Big\{A\in W^{1,\infty}(Q): \ \lVert A\rVert_{\infty}\leq \frac{1}{9R}\Big\}.
\end{align}
{We first prove the uniqueness result for time-independent convection coefficient  $A\in \mathcal{A}$ and time-dependent density  coefficient $q$.  More precisely we prove the following theorem:}

\begin{theorem}\label{Main Theorem}
	Let $\left({A}^{(1)},q_{1}\right)$ and $\left({A}^{(2)},q_{2}\right)$ be two sets of coefficients 
	such that  {$A^{(i)}\in\mathcal{A}$ are time-independent} and {$q_{i}\in L^{\infty}(Q)$} for $i=1,2$. Let $u_{i}$ be the solutions to \eqref{definition of operator} when $\lb A,q\rb=\lb A^{(i)},q_{i}\rb$ and  $\Lambda_{A^{(i)},q_{i}}$ for $i=1,2$ be the DN maps  defined by \eqref{DN map} corresponding to $u_{i}$. Now if
	\begin{align}\label{Equality of DN map}
		\Lambda_{A^{(1)},q_{1}}(f)|_{G}=\Lambda_{A^{(2)},q_{2}}(f)|_{G},\ \mbox{for}\  f\in L^{2}\lb 0,T;H^{1/2}(\PD\Omega)\rb
	\end{align}
	then
	there exists a function $\Phi\in W^{2,\infty}_{0}(\Omega)$ such that 
	
	\[{A}^{(1)}(x)-{A}^{(2)}(x)=\nabla_{x}\Phi(x), \ \ x\in\Omega\] and
	\[q_{1}(t,x)=q_{2}(t,x),\ (t,x)\in Q\]
	provided $A^{(1)}(x)=A^{(2)}(x)$ for $x\in \PD\Omega$.
\end{theorem}
In Theorem \ref{Main Theorem} if we {take} some extra  assumption on convection term $A^{(i)}$ then we can prove the uniqueness {result} for {full recovery of} $A^{(i)}$ even for the case when  {$A^{(i)}\in\mathcal{A}$ for $i=1,2$} are time-dependent. The precise statement of this is given in the following Corollary.
\begin{corollary}\label{Corollary}
	Let $\left({A}^{(1)},q_{1}\right)$ and $\left({A}^{(2)},q_{2}\right)$ be two sets of time-dependent  coefficients 
	such that  $A^{(i)}\in\mathcal{A}$ and $q_{i}\in L^{\infty}(Q)$ for $i=1,2$. Let $u_{i}$ be the solutions to \eqref{definition of operator} when $\lb A,q\rb=\lb A^{(i)},q_{i}\rb$ and  $\Lambda_{A^{(i)},q_{i}}$ for $i=1,2$ be the DN maps  defined by \eqref{DN map} corresponding to $u_{i}$. Now if 	\begin{align}\label{Divergence free}
		\nabla_{x}\cdot A^{(1)}(t,x)=\nabla_{x}\cdot A^{(2)}(t,x),\ (t,x)\in Q
	\end{align} 
	and 
	\[\Lambda_{A^{(1)},q_{1}}(f)|_{G}=\Lambda_{A^{(2)},q_{2}}(f)|_{G},\ f\in L^{2}\lb 0,T;H^{1/2}(\PD\Omega)\rb\]
	then we have 
	\[A^{(1)}(t,x)=A^{(2)}(t,x) \ \mbox{and}\ q_{1}(t,x)=q_{2}(t,x),\ (t,x)\in Q\]
	provided $A^{(1)}(t,x)=A^{(2)}(t,x)$ for $(t,x)\in \Sigma$.
\end{corollary}
\begin{remark}
	The additional assumption \eqref{Divergence free} on convection term $A^{(i)}$ in Corollary \ref{Corollary} have been considered in prior works as well. See for example  \cite{Bellassoued_Kian_Soccorsi_Scrodinger_infinite,Christofol_Soccorsi_magnetic} for the determination of vector field term appearing in the dynamical Schr\"odinger equation and also in \cite{Caro_Kian_Convection_nonlinear} for non-linear parabolic equation. 
\end{remark}

\section{Boundary Carleman estimate}\label{Boundary Carleman estimate}
In this section  we prove a Carleman estimate involving the boundary terms  for the operator $ \la_{A,q} $. 
We  will use this estimate to control  the boundary terms appearing in integral identity given by \eqref{Integral identity before using GO} where no information is given.
\begin{theorem}\label{boundary}
	Let $\varphi(t,x)=\lambda^{2}t+\lambda \omega\cdot x$ where $\omega\in \Sb^{n-1}$ is fixed. Let 
	$ u\in C^2(\overline{Q}) $ such that \[u(0,x)=0, \ \mbox{for} \ x\in\Omega \ \mbox{and} \  u(t,x)=0,\ \mbox{for} \ (t,x)\in \Sigma.\]
	If $A\in\mathcal{A}$ and {$q\in L^{\infty}(Q)$} then there exists $ C >0$ depending only on $ \Omega,T,q$ and $ A $  such that 
	\begin{equation}
		\begin{aligned}\label{bdryesti}
			&\lambda^2\int_{Q} e^{-2\varphi} \lvert u(t,x)\rvert^2 \D x\D t+\int_{Q} e^{-2\varphi} \lvert \nabla_{x}u(t,x)\rvert^2 \d x\d t +\int_{\Omega}  e^{-2\varphi(T,x)}\lvert u(T,x)\rvert^2 \d x \\
			&\ \ + \lambda \int_{\Sigma_{+,\omega}}  e^{-2\varphi} \lvert\PD_{\nu} u(t,x)\rvert^2\lvert \omega\cdot\nu(x)\rvert \d S_{x}  \d t \leq	 C\int\limits_{Q} e^{-2\varphi}\lvert\Lc_{A,q}u(t,x)\rvert^2 \D x \D t\\
			&\ \ \ \ \  \ \  \ \quad\quad \quad \quad  +C\lambda \int\limits_{\Sigma_{-,\omega}} e^{-2\varphi}\lvert\PD_{\nu} u(t,x)\rvert^2 \ \lvert\omega\cdot\nu(x)\rvert  \d S_{x}  \d t
		\end{aligned}
	\end{equation}
	holds for $ \lambda$ large.
\begin{proof}
	Let 
	\begin{align}\label{Definition of conjugated operator}
		\la_{\varphi} :=e^{-\varphi} \la_{A,q} e^{\varphi} 
	\end{align}
	and denote by 
	\begin{align*}
		\widetilde{q}(t,x):=q(t,x)-\nabla_{x}\cdot {A(t,x)-\lvert A(t,x)\rvert^{2}}.
	\end{align*}
	Then
	\begin{align}\label{Expression for Lphi}
		\begin{aligned}
			\la_{\varphi}v(t,x)&=e^{-\varphi} \lb \PD_{t} -\Delta-2{A(t,x)}\cdot\nabla_{x}+\widetilde{q}(t,x)\rb(e^{\varphi}v(t,x)) \\
			&= \lb\PD_{t}- \Delta  -2\nabla_{x} \varphi\cdot\nabla_{x}\rb v(t,x) +\lb \PD_{t}\varphi -\lvert \nabla _{x}\varphi\rvert^2 -\Delta \varphi\rb v(t,x)\\
			&\ \ \ \ -2\lb {A(t,x)}\cdot\nabla_{x}-2{A(t,x)}\cdot\nabla_{x}\phi\rb v(t,x):= \lb P_1 v + P_2 v+P_{3}v\rb(t,x)
		\end{aligned}
	\end{align}
	where
	\begin{align}\label{Definition of Pi}
		\begin{aligned}
			&P_{1} := -\Delta +\PD_{t}\varphi -\lvert \nabla \varphi\rvert^2 -\Delta \varphi=-\Delta \\ 
			&P_2 := \PD_{t}-2\nabla_{x}\varphi\cdot\nabla_{x}=\PD_{t}- 2  \lambda\omega\cdot\nabla_{x}\\ 
			&P_3:= -2{A(t,x)}\cdot\nabla_{x}-2{A(t,x)}\cdot\nabla_{x}\varphi+\widetilde{q}(t,x) =-2{A(t,x)}\cdot\nabla_{x}-2\lambda\omega\cdot {A(t,x)}+\widetilde{q}(t,x).
		\end{aligned}
	\end{align}
	Now let 
	\begin{align}\label{Definition of I L2 norm of conjugated operator}
		\begin{aligned}
			I&:=\int\limits_{Q}\lvert \la_{\varphi} v(t,x)\rvert^{2}\D x\D t\geq \frac{1}{2}\int\limits_{Q}\lvert \lb P_{1}+P_{2}\rb v(t,x)\rvert^{2}\D x\D t-\int\limits_{Q}\lvert P_3 v(t,x)\rvert^{2}\D x\D t:=I_{1}-I_{2}
		\end{aligned}
	\end{align}
	where 
	\begin{align*}
		\begin{aligned}
			I_{1}:=\frac{1}{2}\int\limits_{Q}\lvert \lb P_{1}+P_{2}\rb v(t,x)\rvert^{2}\D x\D t \ \mbox{and}\ 
			I_{2}:= \int\limits_{Q}\lvert P_3 v(t,x)\rvert^{2}\D x\D t.
		\end{aligned}
	\end{align*}
	Next we estimate each of $I_{j}$ for $j=1,2$.  Now $I_{1}$ is 
	\begin{align*}
		I_{1}= \frac{1}{2}\int\limits_{Q} \lvert P_1v(t,x)\rvert^2 + \frac{1}{2}\int_{Q} \lvert P_2v(t,x)\rvert^2 \D x \D t
		+\int\limits_{Q} P_1v(t,x) \ P_2 v(t,x)\D x\D t.
	\end{align*}
	We consider each term separately on right hand side of the above equation.
	Using integration by parts and the fact that $v|_{\Sigma} =0  $, we have
	\begin{align*}
		\int\limits_{Q} \lvert \nabla_{x}v (t,x)\rvert^2 \d x \d t&=-\int\limits_{Q} v (t,x)\Delta v(t,x)
		\d x\d t\leq \frac{1}{2s}\int\limits_{Q} \lvert \Delta v(t,x)\rvert^2 \D x\D t+ \frac{s}{2}\int_{Q}\lvert v(t,x)\rvert^2\D x\D t
	\end{align*}
	holds for any $s>0$.
	Thus we have 
	\begin{align}\label{estimap1}
		\int\limits_{Q} \lvert P_1 v (t,x)\rvert^2 \D x\D t\geq 2s \int\limits_{Q}  \lvert \nabla v (t,x)\rvert^2 \D x\D t- s^2	\int\limits_{Q} \lvert v(t,x)\rvert^2\D x\D t.
	\end{align}
	Following the proof of \cite[Lemma $ 3.1 $]{Choulli_Kian_parabolic_partial} we have 
	\begin{align}\label{estimep2}
		\begin{aligned}
			\int\limits_{Q}\lvert P_{2}v(t,x)\rvert^{2}\D x\D t\geq \frac{1+4\lambda^{2}}{16 R^{2}}\int\limits_{Q}\lvert v(t,x)\rvert^{2}\D x\D t
		\end{aligned}
	\end{align}
	where $R>0$ is the radius of smallest ball $B(0,R)\subset \Rb^{1+n}$ such that $\overline{Q}\subseteq B(0,R)$. Now consider 
	\begin{align}\label{estimap1p2}
		\begin{aligned}
			2 \int_{Q} P_1v(t,x) \ P_2v(t,x)\d x \d t  &= -2\int\limits_{Q} \Delta v(t,x) \PD_{t}v(t,x)\d x \d t + 2\lambda \int\limits_{Q}\Delta v(t,x)\omega\cdot\nabla_{x}v(t,x)\d x \d t\\
			&= \int\limits_{\Omega} \lvert \nabla v(T,x)\rvert^2 \d x +\lambda\int\limits_{\Sigma}\omega\cdot\nu(x)\lvert\PD_{\nu}v(t,x)\rvert^{2}\D S_{x}\D t.
		\end{aligned}
	\end{align}
	Combining Equations  \eqref{estimap1},\eqref{estimep2} and \eqref{estimap1p2} we get
	\begin{equation}\label{estimateforprinci}
		\begin{aligned}
			I_{1}      &\geq s \int\limits_{Q}  \lvert \nabla v (t,x)\rvert^2 \D x\D t- \frac{s^2}{2}	\int\limits_{Q} \lvert v(t,x)\rvert^2\D x\D t +\frac{1+4\lambda^{2}}{32 R^{2}}\int\limits_{Q}\lvert v(t,x)\rvert^{2}\D x\D t\\
			&\ \ \  +\frac{1}{2}\int\limits_{\Omega} \lvert\nabla v(T,x)\rvert^2 \d x +\frac{\lambda}{2}\int\limits_{\Sigma}\omega\cdot\nu(x)\lvert\PD_{\nu}v(t,x)\rvert^{2}\D S_{x}\D t.\\
		\end{aligned}
	\end{equation}
	Next  we estimate $I_2$.  
	\begin{align}\label{estimap3}
		\begin{aligned}
			I_{2}&\leq\int\limits_{Q} \Big\lvert \Big(-2{A(t,x)}\cdot\nabla_{x}-2\lambda\omega\cdot {A(t,x)}+\widetilde{q}(t,x)\Big)v(t,x)\Big\rvert^{2}\D x\D t\\
			& \	\leq 2\lVert \widetilde{q}\rVert^2_{\infty}\int_{Q}\lvert v(t,x)\rvert^2\d x \d t +8\lambda^2 \lVert A \rVert^2_{\infty} \int_{Q}\lvert v(t,x)\rvert^2 \d x \d t+ 8\lVert A \rVert^2_{\infty} \int\limits_{Q} \lvert\nabla v(t,x)\rvert^2 \d x\d t.
		\end{aligned}
	\end{align}
	Using \eqref{estimateforprinci} and \eqref{estimap3} in \eqref{Definition of I L2 norm of conjugated operator} we get 
	\begin{align*}
		\begin{aligned}
			I&\geq s \int\limits_{Q}  \lvert \nabla_{x} v (t,x)\rvert^2 \D x\D t- \frac{s^2}{2}	\int\limits_{Q} \lvert v(t,x)\rvert^2\D x\D t +\frac{1+4\lambda^{2}}{32 R^{2}}\int\limits_{Q}\lvert v(t,x)\rvert^{2}\D x\D t\\
			&\ \ \ +\frac{1}{2}\int\limits_{\Omega} \lvert\nabla_{x}v(T,x)\rvert^2 \d x  +\frac{\lambda}{2}\int\limits_{\Sigma}\omega\cdot\nu(x)\lvert\PD_{\nu}v(t,x)\rvert^{2}\D S_{x}\D t - 2\lVert \tilde{q}\rVert^2_{\infty}\int_{Q}\lvert v(t,x)\rvert^2\d x \d t\\
			& \ \ \  -8\lambda^2 \lVert A \rVert^2_{\infty} \int_{Q}\lvert v(t,x)\rvert^2 \d x \d t
			-8\lVert A \rVert^2_{\infty} \int\limits_{Q} \lvert \nabla_{x}v(t,x)\rvert^2 \d x\d t.\\
			&\geq \lb \frac{1+4\lambda^{2}}{32 R^{2}}- \frac{s^2}{2}-2\lVert \tilde{q}\rVert^2_{\infty}-8\lambda^2 \lVert A \rVert^2_{\infty}\rb\int\limits_{Q}\lvert v(t,x)\rvert^{2}\D x \D t+\lb s -8\lVert A \rVert^2_{\infty} \rb\int\limits_{Q}  \lvert \nabla_{x} v (t,x)\rvert^2 \D x\D t\\
			&\ \ \ \ +\frac{1}{2}\int\limits_{\Omega} \lvert\nabla_{x} v(T,x)\rvert^2 \d x +\frac{\lambda}{2}\int\limits_{\Sigma}\omega\cdot\nu(x)\lvert\PD_{\nu}v(t,x)\rvert^{2}\D S_{x}\D t.
		\end{aligned}
	\end{align*}
	Now since $\lVert A\rVert_{\infty}\leq \frac{1}{9R}$, therefore taking $\lambda$ large enough and using the Poincar\' e inequality, we have 
	\begin{align}\label{main estimate}
		\begin{aligned}
			\int\limits_{Q}\lvert \la_{\varphi} v(t,x)\rvert^{2}\D x\D t &\geq C \Bigg( \lambda^{2}\int\limits_{Q}\lvert v(t,x)\rvert^{2}\D x\D t+\int\limits_{Q}  \lvert \nabla_{x} v (t,x)\rvert^2 \D x\D t\\
			&\qquad +\frac{1}{2}\int\limits_{\Omega} \lvert v(T,x)\rvert^2 \d x + \lambda\int\limits_{\Sigma}\omega\cdot\nu(x)\lvert\PD_{\nu}v(t,x)\rvert^{2}\D S_{x}\D t \Bigg)
		\end{aligned}
	\end{align}
	holds for large  $\lambda$ and  $C>0$ depending only on $Q$, $A$ and $\widetilde{q}$.
	Now after substituting $ v(t,x) = e^{-\varphi(t,x)} u(t,x) $ in \eqref{main estimate}, we get 	
	\begin{align*}
		\begin{aligned}
			&\lambda^2\int_{Q} e^{-2\varphi} \lvert u(t,x)\rvert^2 \D x\D t+\int_{Q} e^{-2\varphi} \lvert\nabla_{x}u(t,x) \rvert^2 \d x\d t +\int_{\Omega}  e^{-2\varphi(T,x)}\lvert u(T,x)\rvert^2 \d x \\
			&\ \ + \lambda \int_{\Sigma_{+,\omega}}  e^{-2\varphi} \lvert\PD_{\nu} u(t,x)\rvert^2\lvert \omega\cdot\nu(x)\rvert \d S_{x}  \d t \leq	 C\int\limits_{Q} e^{-2\varphi}\lvert\Lc_{A,q}u(t,x)\rvert^2 \D x \D t\\
			&\ \ \ \ \  \ \  \ \quad\quad \quad \quad  +C\lambda \int\limits_{\Sigma_{-,\omega}} e^{-2\varphi}\lvert\PD_{\nu} u(t,x)\rvert^2 \ \lvert\omega\cdot\nu(x)\rvert  \d S_{x}  \d t.
		\end{aligned}
	\end{align*}
	This completes the proof of Carleman estimate given by \eqref{boundary}.
\end{proof}
\end{theorem}
\section{Construction of geometric optics solutions}\label{Contruction of go solutions}
In this section, we  construct the exponentially growing solution to 
\[\Lc_{A,q}u(t,x)=0,\ (t,x)\in Q\]
and exponentially decaying solution to 
\[\la^{*}_{A,q}u(t,x)=0,\ (t,x)\in Q\]
where $\la^{*}_{A,q}$ given by
\[\la^{*}_{A,q}:=  -\PD_{t} - \sum_{j=1}^{n}\lb\PD_{j} -A_{j}(t,x)\rb^2 +{\widetilde{q}^{*}}(t,x)\]
is	a formal $L^{2}$ adjoint of the operator $\la_{A,q}$.
We construct these solutions by using a Carleman estimate in a Sobolev space of negative order as used in \cite{Ferriera_Kenig_Sjostrand_Uhlmann_magnetic} for elliptic case and in \cite{Kian_damping,Krishnan_Vashisth_Relativistic} for hyperbolic case. Before going further following \cite{Kian_damping} we will give some definition and notation, which will be used later. 	For $m\in \Rb$, define space $ L^2\lb0,T;H^m_{\lambda}(\Rn)\rb$  by
\begin{align*}
	L^2\lb0,T;H^m_{\lambda}(\Rn)\rb:= \left\{u(t,\cdot)\in \mathcal{S}^{'}(\Rn):\lb\lambda^2+|\xi|^2\rb^{m/2}\widehat{u}(t,\xi) \in L^2(\Rn) \right\}
\end{align*}
with the norm
\[ \lVert u \rVert^2_{L^2\lb0,T;H^m_{\lambda}(\Rn)\rb}:= \int\limits_{0}^{T}\int\limits_{\Rn} \lb \lambda^2+\lvert\xi\rvert^2\rb^{m} \lvert\widehat{u}(t,\xi)\rvert^2 d\xi dt\]
where $\mathcal{S}^{'}(\Rn) $ denote the space of all tempered distribution on $\Rn$ and $\widehat{u}(t,\xi)$ is the Fourier transform  with respect to space variable $x\in \Rb^{n}$.
We define by 
\[\langle D_{x},\lambda \rangle^{m} u = \mathcal{F}_{x}^{-1} \left\{ \lb\lambda^2+|\xi|^2\rb^{m/2} \mathcal{F}_{x}u  \right\}\]
here $\mathcal{F}_{x}$ and $\mathcal{F}_{x}^{-1}$ denote the Fourier transform and inverse Fourier transform respectively with respect to space variable $x\in\Rb^{n}$. With this we define the symbol class $S^{m}_{\lambda}(\Rn)$ of order $m$ by
\[S^{m}_{\lambda}(\Rn):=\left\{ C_{\lambda} \in C^{\infty}(\Rn \times \Rn) : \lvert \PD^{\alpha}_{x}\PD^{\beta}_{\xi} c_{\lambda}(x,\xi)\rvert\leq C_{\alpha,\beta} \lb\lambda^2+ \lvert\xi\rvert^2\rb^{m-\lvert\beta\rvert}, \ \mbox{for multi-indices } \alpha,\beta \in \mathbb{N}^n\right\}.\]
With these notations and definitions, we state the main theorem of this section.
\begin{theorem}\label{Construction of solution}
	\begin{enumerate}
		\item $($Exponentially growing solutions$)$ Let  $\Lc_{A,q}$ be as defined above. Then for $\lambda$ large there exists {$ v \in H^{1}\lb 0,T; H^{-1}(\Omega)\rb\cap L^{2}\lb0,T;H^{1}(\Omega)\rb $} a solution to 
		\begin{align*}
			\begin{aligned}
				\begin{cases}
					&\la_{A,q}  v(t,x)=0,\ (t,x)\in Q,\\
					&v(0,x)=0,\ x\in\Omega
				\end{cases}
			\end{aligned} 
		\end{align*}
		of the  following form
		\begin{align}\label{growingsolu}
			v_{g}(t,x)=  e^{\varphi}\Big(B_{g}(t,x) +R_{g}(t,x,\lambda)  \Big)
		\end{align}
		where for $\chi\in C_{c}^{\infty}((0,T))$ arbitrary, we have 
		\begin{align}\label{expreofbg}
			B_{g}(t,x)=\chi(t) e^{-\I (t\tau+x\cdot \xi)}\exp\bigg(\int_{0}^{\infty}\omega \cdot {A(t,x+s\omega)}\d s \bigg)
		\end{align}
		and $R_{g}(t,x,\lambda)$ satisfies the following
		\begin{align}\label{estimofrg}
			R_{g}(0,x,\lambda)=0,\ \mbox{for $x\in\Omega$ and}\ 	{\lVert R_{g}\rVert_{L^{2}\lb 0,T;H^1_{\lambda}(\Rn)\rb} \leq C}.
		\end{align}
		\item $($Exponentially decaying solutions$)$ Let $\la^{*}_{A,q}$ be as before. Then for $\lambda$ large there exists {$ v \in H^{1}\lb 0,T; H^{-1}(\Omega)\rb\cap L^{2}\lb0,T;H^{1}(\Omega)\rb $} a solution to 
		\begin{align*}
			\begin{aligned}
				\begin{cases}
					&\la^{*}_{A,q}  v(t,x)=0,\ (t,x)\in Q,\\
					&v(T,x)=0,\ x\in\Omega
				\end{cases}
			\end{aligned} 
		\end{align*}
		of  the following form
		\begin{align}\label{decayingsolu}
			v_{d}(t,x)=  e^{-\varphi}\Big( B_{d}(t,x) +R_{d}(t,x,\lambda)  \Big)
		\end{align}
		where for $\chi\in C_{c}^{\infty}((0,T))$ arbitrary, we have  
		\begin{align}\label{expreofbd}
			B_{d}(t,x)=\chi(t) \exp\bigg(-\int_{0}^{\infty}\omega \cdot {A(t,x+s\omega)}\d s \bigg)
		\end{align}
		and $R_{d}(t,x,\lambda)$ satisfies the following 
		\begin{align}\label{estimofrd}
			R_{d}(T,x,\lambda)=0,\ \mbox{for $x\in\Omega$ and}\ 	{\lVert R_{d}\rVert_{L^{2}\lb 0,T;H^{1}_{\lambda}(\Rn)\rb} \leq C.}
		\end{align}
	\end{enumerate}
\end{theorem}
Proof of the above theorem is based on a Carleman estimate in a Sobolev space of negative order. To prove the Carleman estimate stated in Proposition \ref{Carleman estimate of negative order}, we follow the  arguments similar to  one used in \cite{Ferriera_Kenig_Sjostrand_Uhlmann_magnetic,Kian_damping,Krishnan_Vashisth_Relativistic} for elliptic and hyperbolic inverse problems. 

\begin{proposition}\label{Carleman estimate of negative order}
	Let  $ \varphi, A$ and $q$ be as in Theorem \ref{boundary}. Then for $\lambda$ large enough, we have 
	\begin{enumerate}
		\item $($Interior Carleman  estimate for $\la^{*}_{\varphi}$ $)$  Let $\Lc_{\varphi}^{*}:=e^{\varphi}\Lc_{A,q}^{*}e^{-\varphi}$, then there exists a constant  $C>0$ independent of $\lambda$ and $v$ such that
		\begin{align}\label{intericarleadj}
			\lVert v \rVert_{\lt}\leq  C \lVert \la^{*}_{\varphi} v\rVert_{\lth}, 
		\end{align}
		holds for $v\in C^{1}\lb [0,T];C_{c}^{\infty}(\Omega)\rb$ satisfying $v(T,x)=0$ for $x\in \Omega$.
		\item $($Interior Carleman estimate for $\la_{\varphi}$ $)$ Let $\Lc_{\varphi}$ be as before then there exists a constant  $C>0$ independent of $\lambda$ and $v$ such that
		\begin{align}\label{intericarle}
			\lVert v \rVert_{\lt}\leq  C \lVert\la_{\varphi} v\rVert_{\lth}
		\end{align}
		holds for  $v\in C^{1}\lb [0,T];C_{c}^{\infty}(\Omega)\rb$ satisfying $v(0,x)=0$ for $x\in \Omega$. 
	\end{enumerate}
	\begin{proof}
		\begin{enumerate}
			\item $($Proof for \eqref{intericarleadj}$)$ Since
			\begin{align*}
				\la_{\varphi}^{*}=	e^{\varphi} \la^{*}_{A,q} e^{-\varphi} 
			\end{align*}
			therefore we have 
			\begin{align*}
				\la_{\varphi}^{*}v&=e^{\varphi}\Big(-\PD_{t}-\Delta+2{A(t,x)}\cdot\nabla_{x}+\widetilde{q}^{*}(t,x)\Big) e^{-\varphi}v(t,x)\\
				&=\Big[-\PD_{t}-\Delta+2{A(t,x)}\cdot\nabla_{x}+\widetilde{q}^{*}(t,x)-2\lambda \omega\cdot {A(t,x)}+2\lambda\omega\cdot\nabla_{x}\Big]v(t,x)\\
			\end{align*}
			where 
			\begin{align*}
				\widetilde{q}^{*}(t,x):= \overline{q}(t,x)+\nabla_{x}\cdot {A(t,x)}-\lvert {A(t,x)}\rvert^{2}.
			\end{align*}
			Writing  $\Lc_{\varphi}^{*}$ as 
			\[\la_{\varphi}^{*}v:= P_{1}^{*}v+P_{2}^{*}v+P_{3}^{*}v\]
			where 
			\begin{align}\label{Definition of Pi*}
				\begin{aligned}
					P_{1}^{*}:=-\Delta,  \ P_{2}^{*}:=-\PD_{t}+2\lambda\omega\cdot\nabla_{x} \ \mbox{and}\ P_{3}^{*}:=2{A(t,x)}\cdot\nabla_{x}-2\lambda\omega\cdot {A(t,x)}+\widetilde{q}^{*}(t,x).
				\end{aligned}
			\end{align}
			Now from \eqref{Definition of Pi}, we have $	P_{1}^{*}=P_{1} \ \mbox{and} \ P_{2}^{*}=-P_{2}.$
			Hence using the arguments similar to Theorem \ref{boundary}, we have
			\begin{align*}
				\int\limits_{Q} \lvert\nabla v(t,x) \rvert^2 \d x\d t+\lambda^2\int\limits_{Q}\lvert v(t,x)\rvert^2 \d x\d t\leq C\int_{Q} \lvert\la^*_{\varphi}v(t,x)\rvert^2 \d x\d t
			\end{align*}
			for some constant $C>0$ independent of $\lambda$ and $v$.
			The above estimate can be written in compact form as 
			\begin{align}\label{main estimateadj}
				\lVert v \rVert_{L^{2}\lb 0,T;H^{1}_{\lambda}(\Rn)\rb} \leq C \lVert \la^{*}_{\varphi}v\rVert_{L^{2}(Q)},\
			\mbox{for some constant $C$ independent of $\lambda$ and $v$.}
			\end{align}
			Next  using the pseudodifferential operators techniques, we  shift the index by $-1$ in the above estimate. {Let us denote by $ \widetilde{\Omega} $ a bounded open subset of $ \Rn $ such that $ \overline{\Omega} \subset \widetilde{\Omega} $.  Fix $w\in C^{1}\lb [0,T];C_{c}^{\infty}(\Omega)\rb$ satisfying $w(T,x)=0$ and  consider the following }
			\[ \langle D_{x},\lambda \rangle^{-1} ( P^*_1+P^*_2) \langle D_{x},\lambda \rangle w.\]
			Using the composition of pseudodifferential operators \cite[Theorem 18.1.8]{Hormander} we have
			\begin{align}\label{pseudocompositionadj}
				\langle D_{x},\lambda \rangle^{-1} ( P^*_1+P^*_2) \langle D_{x},\lambda \rangle w = (P^*_1 +P^*_2 ) w.
			\end{align}
			Using \eqref{pseudocompositionadj} and  \eqref{estimateforprinci} we have 
			\begin{align}\label{interiorestiadj}
				\begin{aligned} 
					&\lVert (P^*_1+P^*_2) \langle D_{x},\lambda \rangle w\rVert_{\lth}= \lVert \langle D_{x},\lambda \rangle^{-1} (P^*_1+P^*_2) \langle D_{x},\lambda \rangle w\rVert_{\lt}\\
					&= \lVert (P^*_1+P^*_2)w \rVert_{\lt}\geq \sqrt{s}  \lVert \nabla w\rVert_{\lt} +\lambda \lVert w\rVert_{\lt}
				\end{aligned}
			\end{align}
			holds for $\lambda$ large.	
			Now consider 
			\begin{align*}
				\begin{aligned}
					&\lVert P_{3}^{*}  \langle D_{x},\lambda \rangle w\rVert_{\lth} \leq 2 \bigg( \lVert (  \lambda \omega\cdot {A(t,x)}  \langle D_{x},\lambda \rangle w \rVert_{\lth}\\
					&\ \ \ \ \ \ \ \ \ + \lVert {A(t,x)}\cdot\nabla_{x} \langle D_{x},\lambda \rangle w \rVert_{\lth} +\lVert \widetilde{q}^{*}  w \rVert_{\lth}\bigg).
				\end{aligned}
			\end{align*}
			Using the boundedness of the coefficients, we have
			\begin{align*}
				\begin{aligned}
					&\lVert P_{3}^{*}  \langle D_{x},\lambda \rangle w\rVert_{\lth} \leq  2 \bigg( \lambda \lVert A \rVert_{\infty}\lVert w \rVert_{\lt}\\
					&\ \ \ \ \qquad\qquad \qquad\qquad \qquad + \lVert A \rVert_{\infty}\lVert \nabla w \rVert_{\lt} + \lVert\tilde{q}\rVert_{\infty} \lVert w \rVert_{\lt} \bigg).
				\end{aligned}
			\end{align*}
			Hence using the inequality as used in  \eqref{Definition of I L2 norm of conjugated operator} we get 
			\begin{align*}
				\lVert \la^*_{\varphi} \langle D,\lambda \rangle w\rVert_{\lth}
				\geq C \lVert w\rVert_{\L^{2}\lb 0,T;H^{1}_{\lambda}(\Rn)\rb}.
			\end{align*}
			{Now let $ \chi \in C^{\infty}_{c}(\widetilde{\Omega})$ such that $ \chi =1 $ in $\overline{\Omega}_{1}$ where $ \overline{\Omega} \subset \Omega_1 \subset \widetilde{\Omega}$.} Fix $ w= \chi \langle D,\lambda \rangle^{-1}v$ in the above equation and using  
			\[ \lVert (1-\chi) \langle D,\lambda \rangle^{-1}v \rVert_{L^{2}(0,T;H^{m}_{\lambda}(\Rn))} \leq \frac{C}{\lambda^2} \lVert v \rVert_{L^{2}\lb 0,T;L^2(\Rn)\rb}\]
			and 
			\begin{align*}
			\begin{aligned}
				\lVert v \rVert_{\lt}&= \lVert \langle D,\lambda \rangle^{-1} v \rVert_{L^{2}(0,T;H^{1}_{\lambda}(\Rn))}\\
				&\leq \lVert w\rVert_{L^{2}(0,T;H^{1}_{\lambda}(\Rn))}+\lVert (1-\chi) \langle D,\lambda \rangle^{-1}v \rVert_{L^{2}(0,T;H^{1}_{\lambda}(\Rn))}\\
				&\leq \lVert w\rVert_{L^{2}(0,T;H^{1}_{\lambda}(\Rn))} +\frac{C}{\lambda^2} \lVert v \rVert_{L^2(0,T;L^2(\Rn))}
				\end{aligned}
			\end{align*}
			we get
			\begin{align*}
			\begin{aligned}
				\lVert \la_{\varphi}^* v\rVert_{\lth}&\geq \lVert \la_{\varphi}^* \langle D,\lambda \rangle w\rVert_{\lth} - \frac{C}{\lambda^2} \lVert v \rVert_{L^2(0,T;L^2(\Rn))}\\
				&\geq  \lVert w\rVert_{L^{2}(0,T;H^{1}_{\lambda}(\Rn))} -\frac{C}{\lambda^2} \lVert v \rVert_{L^2(0,T;L^2(\Rn))}\geq C \lVert v\rVert_{\lt} 
				\end{aligned}
			\end{align*}
for  large $ \lambda $.
			Thus finally, we have 
				\[\lVert v\rVert_{\lt} \leq C 	\lVert \la_{\varphi}^* v\rVert_{\lth}\]
				holds for  $\ v\in C^{1}\lb [0,T];C_{c}^{\infty}(\Omega)\rb$ such that $v(T,x)=0$ and  $\lambda$  large.
			\item (Proof for \eqref{intericarle})  follows  by exactly the same argument as that for \eqref{intericarleadj}.
		\end{enumerate}
	\end{proof}	
\end{proposition}

\begin{proposition}\label{exisofsolu}
	Let $ \varphi $, $A $ and $ q $ be as in Theorem \ref{boundary}. 
	\begin{enumerate}
		\item $($Existence of solution to $\la_{A,q}$ $)$ For $ \lambda > 0 $ large enough and {$ v\in L^{2}(Q)$} there exists a solution  {$  u\in H^{1}\lb 0,T; H^{-1}(\Omega)\rb\cap L^{2}\lb0,T;H^{1}(\Omega)\rb$} of 
		\begin{align*}
			\begin{aligned}
				\begin{cases}
					&\la_{\varphi} u(t,x)=v(t,x),\ (t,x)\in Q,\\
					&u(0,x)=0;\ x\in\Omega
				\end{cases}
			\end{aligned}
		\end{align*}
		and it satisfies 
		\begin{align}\label{Estimate for existence}
			{\lVert u \rVert_{L^{2}\lb 0,T;H^{1}(\Omega)\rb} \leq C \lVert v \rVert_{L^{2}(Q)}}
		\end{align}
		where $ C > 0  $ is a constant independent of $\lambda$.
		\item $($Existence of solution to $\la^*_{A,q}$$)$ 	For $ \lambda > 0 $ large enough and {$ v\in L^{2}(Q)$} there exists a solution  {$  u\in H^{1}\lb 0,T; H^{-1}(\Omega)\rb\cap L^{2}\lb0,T;H^{1}(\Omega)\rb$} of 
		\begin{align*}
			\begin{aligned}
				\begin{cases}
					&\la^{*}_{\varphi} u(t,x)=v(t,x),\ (t,x)\in Q,\\
					&u(T,x)=0,\ x\in\Omega
				\end{cases}
			\end{aligned}
		\end{align*} 
		and it satisfies
		\begin{align}\label{Estimate for adjoint existence}
			{\lVert u \rVert_{L^{2}\lb 0,T;H^{1}(\Omega)\rb} \leq C \lVert v \rVert_{L^{2}(Q)}}
		\end{align}
		where $ C > 0  $ is a constant independent of $\lambda$.
	\end{enumerate}

	\begin{proof}
		\begin{enumerate}
			We will give the proof for existence of  solution to $\la_{A,q}$ and the proof for $\la^{*}_{A,q}$ follows by using similar arguments. The proof is based on the standard functional analysis arguments. Consider the space $ S:=\{ \la^{*}_{\varphi} u: u\in C^{1}\lb[0,T];C^{\infty}_{c}(\Omega)\rb\ \mbox{and}\ u(T,x)=0\} $ as a subspace of $ \lth $. Define the linear operator  $ T $ on $ S $ by 
			\begin{align*}
				T(\la^{*}_{\varphi} z) = \int_{Q} z(t,x) v(t,x) \d t \d x, \ \ ~ \text{for $ {v \in L^{2}(Q)}$}.
			\end{align*}
			Now using the Carleman estimates \eqref{intericarleadj}, we have
			\begin{align*}
				\lvert T(\la^{*}_{\varphi} z) \rvert \leq \lVert z \rVert_{L^2(Q)} \lVert v \rVert_{L^2(Q)} \leq C  \lVert v \rVert_{L^2(Q)} \lVert \la^{*}_{\varphi} z \rVert_{\lth}
			\end{align*}
			holds for 	$z\in C^{1}\lb [0,T];C_{c}^{\infty}(\Omega)\rb$ with $z(T,x)=0$.	Hence using the Hahn-Banach theorem, we can extend the linear operator $ T $ to $ \lth $. We denote the extended map as $ T $ and it satisfies $$ \lVert T \rVert \leq C \lVert v \rVert_{L^2(Q)}. $$ 
			Since $ T $ is bounded linear functional on $ \lth $ therefore using the  Riesz representation theorem there exists a unique $ u \in L^2(0,T;H^1_{\lambda}(\Rn) ) $ such that
			\begin{align}\label{Exitence of u after using Riesz}
				\ \ \ \ \ \ 	T(f)= \langle f,u \rangle_{{L^{2}\lb 0,T;H^{-1}_{\lambda}(\Rb^{n})\rb,L^{2}\lb 0,T;H^{1}_{\lambda}(\Rb^{n})\rb}}\  \mbox{for}\  f \in \lth
			\end{align}
			with $ \lVert u \rVert_{L^2(0,T;H^1_{\lambda}(\Rn) )} \leq C \lVert v \rVert_{L^2(Q)} $.  Now for $z\in C^{1}\lb[0,T];C_{c}^{\infty}(\Omega)\rb $ satisfying $z(T,x)=0$. Choosing $f=\la_{\varphi}^{*}z$ in the above equation, we get $\la_{\varphi}u=v$.
			{Using the expression for $\la_{\varphi}$ from \eqref{Expression for Lphi} and the fact that $u\in L^{2}\lb 0,T;H^{1}(\Omega)\rb$ and $v\in L^{2}(Q)$,  we get that $\PD_{t}u\in L^{2}\lb 0,T;H^{-1}(\Omega)\rb$. Hence we have $  u\in H^{1}\lb 0,T; H^{-1}(\Omega)\rb\cap L^{2}\lb0,T;H^{1}(\Omega)\rb$.}
			
			Next we will show that $u(0,x)=0$ for $x\in\Omega$. To prove this we choose $f=\la_{\varphi}^{*}z$ where  $z\in C^{1}\lb [0,T];C_{c}^{\infty}(\Omega)\rb$ and $z(T,x)=0$. Using this choice of $f$  in \eqref{Exitence of u after using Riesz}, we have 
			\[\int\limits_{Q}\la_{\varphi}^{*}z(t,x)u(t,x)\D x\D t=\int\limits_{Q}z(t,x)v(t,x)\D x\D t.\]
			Now using integration by parts and the fact that $\la_{\varphi}u=v$, we get 
			\[\int\limits_{\Omega}u(0,x)z(0,x)\D x=0.\] The above identity holds for any $z\in C^{1}\lb [0,T];C_{c}^{\infty}(\Omega)\rb$ satisfying $z(T,x)=0$. Therefore, we conclude that $u(0,x)=0$ for $x\in\Omega$. This completes the proof of first part of Proposition \ref{exisofsolu} .
		\end{enumerate}
	\end{proof}	
\end{proposition}
\subsection{Proof of the Theorem \ref{Construction of solution}}
Using expressions $v_{g}$ and $B_{g}$ from \eqref{growingsolu} and \eqref{expreofbg} respectively and 
\begin{align*}
	\la_{A,q}v_{g}(t,x)=0,
\end{align*}
we have the equation for $R_{g}$ is 
\begin{align*}
	\la_{\varphi}R_{g}(t,x,\lambda)=-\la_{A,q}B_{g}(t,x),
\end{align*}
where {$\la_{A,q}B_{g}(t,x)\in L^{2}(Q)$}. 
Next using Proposition \ref{exisofsolu}, there exists $R_{g}\in  L^{2}\lb 0,T;H^{1}(\Omega)\rb\cap H^{1}\lb 0,T;H^{-1}(\Omega)\rb $ solution to 
\begin{align*}
	\begin{aligned}
		\begin{cases}
			&\la_{\varphi}R_{g}(t,x,\lambda)=-\la_{A,q}B_{g}(t,x),\ (t,x)\in Q,\\
			&R_{g}(0,x,\lambda)=0, \ x\in\Omega
		\end{cases}
	\end{aligned}
\end{align*} 
and it satisfies the following estimate
{\begin{align*}
		\lVert R_{g}\rVert_{L^{2}\lb 0,T;H^{1}(\Omega)\rb}\leq C 
\end{align*}}
where $C$ is a constant independent of $\lambda$. This completes the construction of  solution for  $\Lc_{A,q}u=0$ and  
existence of the solution for $\la^{*}_{A,q}v=0$, follows in a similar way.

\section{Integral identity}\label{Integral identity}
This section is devoted to proving an integral identity which will be used to prove the main result of this article. We derive this identity by using the geometric optics solutions constructed in \S \ref{Contruction of go solutions}.
Let $u_{i}$ be the solutions to the following initial boundary value problems with vector field coefficient $A^{(i)}$ and scalar potential $q_{i}$ for $i=1,2$.
\begin{align}\label{Equation for ui}
	\begin{aligned}
		\begin{cases}
			&\Lc_{{A}^{(i)},q_{i}}u_{i}(t,x)=0,\ (t,x)\in Q\\
			&u_{i}(0,x)=0, \ x\in\Omega\\
			&u_{i}(t,x)=f(t,x),\ (t,x)\in \Sigma.
		\end{cases}
	\end{aligned}
\end{align} 
Let us denote 
\begin{align}\label{Difference defn}
	\notag &u(t,x):=\lb u_{1}-u_{2}\rb(t,x)\\
	\notag &{{A}(t,x):=\lb{A}^{(1)}-{A}^{(2)}\rb(t,x):=\lb A_{1}(t,x),\cdots,A_{n}(t,x)\rb}\\
	&\wt{q}_{i}(t,x):=-\nabla_{x}\cdot{{A}^{(i)}(t,x)-\lvert{A}^{(i)}(t,x)\rvert^{2}}+q_{i}(t,x)\\
	\notag &\wt{q}(t,x):= \wt{q}_{1}(t,x)-\wt{q}_{2}(t,x).
\end{align}
Then $u$ is  solution to the following initial boundary value problem: 
\begin{align}\label{Equation for u linear}
	\begin{aligned}
		\begin{cases}
			&	\Lc_{{A}^{(1)},q_{1}}u(t,x)
			=2{A(t,x)}\cdot\nabla_{x}u_{2}(t,x)+\wt{q}u_{2}(t,x),\ (t,x)\in Q\\
			&	u(0,x)=0,\ x\in\Omega\\
			&u(t,x)=0,\  (t,x)\in\Sigma.
		\end{cases}
	\end{aligned}
\end{align}
Let $v(t,x)$ of the form given by \eqref{decayingsolu}  be the solution to following equation 
\begin{align}\label{adjoint equation for u1}
	\Lc^{*}_{{A}^{(1)},{q}_{1}}v(t,x)
	=0,\  (t,x)\in Q.
\end{align}
Also let $u_{2}$ of the form given by \eqref{growingsolu} be solution to the following equation
\begin{align}\label{equation for u2}
	\begin{aligned}\
		\Lc_{{A}^{(2)},q_{2}}u_{2}(t,x)
		&=0,\   (t,x)\in Q.
	\end{aligned}
\end{align}
Since the right hand side of \eqref{Equation for u linear} lies in $L^{2}(Q)$ therefore using (\cite{Choulli_Book}, Theorem $1.43$) we have $u\in L^{2}\lb 0,T;H^{2}(\Omega)\rb \cap H^{1}\lb0,T;L^{2}(\Omega)\rb$ and $\partial_{\nu}u\in L^{2}\lb0,T;H^{1/2}(\Sigma)\rb.$  Next consider the following
\begin{align*}
	\begin{aligned}
		&\Big\langle \lb \Lambda_{A^{(1)},q_{1}}-\Lambda_{A^{(2)},q_{2}}\rb(f),v|_{\Sigma}\Big\rangle_{\mathcal{H}_{T}^{*},\mathcal{H}_{T}}=\Big\langle \mathcal{N}_{A^{(1)},q_{1}}u_{1}-\mathcal{N}_{A^{(2)},q_{2}}u_{2},v|_{\Sigma}\Big\rangle_{\mathcal{H}_{T}^{*},\mathcal{H}_{T}}\\
		&\ \  =\int\limits_{Q}\lb -u_{1}\PD_{t}\overline{v}+\nabla_{x}u_{1}\cdot\nabla_{x}\overline{v}+2u_{1} A^{(1)}\cdot\nabla_{x}\overline{v}+(\nabla_{x}\cdot A^{(1)}) u_{1}\overline{v}-\lvert A^{(1)}\rvert^{2}u_{1}v+q_{1}u_{1}\overline{v}    \rb\D x\D t\\
		&\ \ \ - \int\limits_{Q}\lb -u_{2}\PD_{t}\overline{v}+\nabla_{x}u_{2}\cdot\nabla_{x}\overline{v}+2u_{2} A^{(2)}\cdot\nabla_{x}\overline{v}+(\nabla_{x}\cdot A^{(2)}) u_{2}\overline{v}-\lvert A^{(2)}\rvert^{2}u_{2}v+q_{2}u_{2}\overline{v}    \rb\D x\D t. 
	\end{aligned}
\end{align*}
After following the arguments  used in $[$\cite{Caro_Kian_Convection_nonlinear}, see Proposition $2.3$$]$, we get that 
\begin{align}\label{Difference of DN maps}
	\begin{aligned}
		&\Big\langle \lb \Lambda_{A^{(1)},q_{1}}-\Lambda_{A^{(2)},q_{2}}\rb(f),v|_{\Sigma}\Big\rangle_{\mathcal{H}_{T}^{*},\mathcal{H}_{T}}=\int\limits_{Q}\lb 2{A(t,x)}\cdot\nabla_{x}u_{2}(t,x)+\wt{q}(t,x)u_{2}(t,x)\rb\overline{v(t,x)}\D x \D t.
	\end{aligned}
\end{align}

Also multiplying \eqref{Equation for u linear} by $\overline{v}(t,x)$ and integrating over $Q$, we have 
\begin{align*}
	\begin{aligned}
		&\int\limits_{Q}\lb 2{A(t,x)}\cdot\nabla_{x}u_{2}(t,x)+\wt{q}(t,x)u_{2}(t,x)\rb\overline{v(t,x)}\D x \D t=\int\limits_{Q}\Lc_{A^{(1)},q_{1}}u(t,x) \overline{v(t,x)} \D x \D t\\
		&\ \ \ =\int\limits_{Q}u(t,x) \overline{\Lc^{*}_{A^{(1)},{q}_{1}}v(t,x)} \D x \D t-\int\limits_{\Sigma}\partial_{\nu}u(t,x)\overline{v(t,x)}\D S_{x}\D t + \int\limits_{\Omega}u(T,x)\overline{v(T,x)}\D x
	\end{aligned}
\end{align*}
where in deriving the above identity we have used the following: $u|_{\Sigma}=0$, $u|_{t=0}=0$ and $A^{(1)}=A^{(2)}$ on $\Sigma$. Now using Equation \eqref{Difference of DN maps} and the fact that $\Lc^{*}_{{A}^{(1)},{q}_{1}}v(t,x)=0$ in $Q$, with $v(T,x)=0$ in $\Omega$,
we get,
\begin{align}\label{Differnce of DN maps equal to integration of Neumann derivative}
	&\Big\langle \lb \Lambda_{A^{(1)},q_{1}}-\Lambda_{A^{(2)},q_{2}}\rb(f),v|_{\Sigma}\Big\rangle_{\mathcal{H}_{T}^{*},\mathcal{H}_{T}}=-\int\limits_{\Sigma}\PD_{\nu}u(t,x)\overline{v(t,x)}\D S_{x}\D t. 
\end{align}
This gives us 
\begin{align}\label{Differnece of DN equal to Neuman derivative}
	\lb \Lambda_{A^{(1)},q_{1}}-\Lambda_{A^{(2)},q_{2}}\rb(f)|_{\Sigma}=-\PD_{\nu}u|_{\Sigma}. 
\end{align}
Using \eqref{Equality of DN map}, we have  $\PD_{\nu}u|_{G}=0$. Finally  using Equations \eqref{Differnce of DN maps equal to integration of Neumann derivative}, \eqref{Differnece of DN equal to Neuman derivative} and $\PD_{\nu}u|_{G}=0$, in \eqref{Difference of DN maps}, we get
\begin{align}\label{Integral identity before using GO}
	\int\limits_{Q}\left(2{A(t,x)}\cdot\nabla_{x}u_{2}(t,x)+\wt{q}(t,x)u_{2}(t,x)\right)\overline{v(t,x)}\D x \D t=
	-\int\limits_{\Sigma\setminus{G}}\partial_{\nu}u(t,x)\overline{v(t,x)}\D S_{x}\D t.
\end{align}
Next we need to estimate the right hand side of  above equation. This we will do in the following lemma:
\begin{lemma}\label{Estimate for remainder lemma}
	Let  $u_{i}$ for $i=1,2$ solutions to \eqref{Equation for ui} with $u_{2}$ of the form \eqref{growingsolu}. Let $u=u_{1}-u_{2}$ and $v$ be of the form \eqref{decayingsolu}. Then 
	\begin{align}\label{Estimate for remainder equation}
		\begin{aligned}
			\Big\lvert\int\limits_{\Sigma\setminus  G}\partial_{\nu}u(t,x)\overline{v(t,x)}\D S_{x}\D t\Big\rvert &\leq  C\lambda^{1/2}
		\end{aligned}
	\end{align}
	for all $\omega\in \Sb^{n-1}$ such that  $\lvert\omega-\omega_{0}\rvert\leq \epsilon$.
\end{lemma}
\begin{proof}
	Using the expression of $v$ from \eqref{decayingsolu}, in the right-hand side of \eqref{Integral identity before using GO}, we have 
	\begin{align*}
		\begin{aligned}
			&\Big\lvert\int\limits_{\Sigma\setminus  G}\partial_{\nu}u(t,x)\overline{v(t,x)}\D S_{x}\D t\Big\rvert^{2} \leq
			\Big\lvert\int\limits_{\Sigma\setminus G}\PD_{\nu}u(t,x)e^{-\vp(t,x)}\lb \overline{B_{d}}(t,x)+\overline{R_{d}}(t,x)\rb\D S_{x}\D t\Big\rvert^{2}\\
			&\qquad \qquad \leq C \left(1+\lVert R_{d}\rVert^2_{L^2(\Sigma)} \right) \int\limits_{\Sigma\setminus{G}}e^{-2\vp(t,x)}\lvert\PD_{\nu}u(t,x)\rvert^{2}\D S_{x}\D t\\
			&\qquad \qquad \quad \leq C\lb 1+\lVert R_{d}\rVert_{{L^{2}\lb 0,T;H^{1}(\Omega)\rb}}\rb\int\limits_{\Sigma\setminus G} e^{-2\vp(t,x)}\lvert\PD_{\nu}u(t,x)\rvert^{2}\D S_{x}\D t
		\end{aligned}		
	\end{align*}
	where in the last step of above inequality we have used the trace theorem. Now 	using Equation \eqref{estimofrd}, we get
	\begin{align*}
		\Big\lvert\int\limits_{\Sigma\setminus  G}\partial_{\nu}u(t,x)\overline{v(t,x)}\D S_{x}\D t\Big\rvert^{2} \leq C \int\limits_{\Sigma\setminus G} e^{-2\vp(t,x)}\lvert\PD_{\nu}u(t,x)\rvert^{2}\D S_{x}\D t.
	\end{align*}
	For $\ve>0$, define 
	\begin{align*}
		\begin{aligned}
			\PD\Omega_{+,\ve,\omega}:=\{x\in \PD \Omega: \nu(x) \cdot \omega > \ve\}\ \mbox{and} \ \Sigma_{+,\ve,\om}:=(0,T)\times \PD \Omega_{+,\ve,\omega}
		\end{aligned}
	\end{align*}
	then from the definition of $G$ it follows that  $\Sigma\setminus{G}\subseteq \Sigma_{+,\ve,\omega}$ for all $\om$ with $\lvert \omega-\omega_{0}\rvert\leq \ve$. Using this we obtain
	\begin{align*}
		\begin{aligned}
			&\int\limits_{\Sigma\setminus{G}}e^{-2\vp(t,x)}\lvert\PD_{\nu}u(t,x)\rvert^{2}\D S_{x}\D t\leq \int\limits_{\Sigma_{+,\ve,\om}}e^{-2\vp(t,x)}\lvert\PD_{\nu}u(t,x)\rvert^{2}\D S_{x}\D t\\
			&\qquad =\frac{1}{\lambda\ve}\int\limits_{\Sigma_{+,\ve,\om}}\lambda\ve	e^{-2\vp(t,x)}\lvert\PD_{\nu}u(t,x)\rvert^{2}\D S_{x}\D t,\ \mbox{for}\  \omega\in\Sb^{n-1} \ \mbox{near}\ \omega_{0}\in\Sb^{n-1}.
		\end{aligned}
	\end{align*}
	Now $\lambda\epsilon\leq \PD_{\nu}\varphi(t,x)$ for $(t,x)\in \Sigma_{+,\epsilon,\omega}$ and $\omega\in \Sb^{n-1}$ with $\lvert \omega-\omega_{0}\rvert\leq \epsilon$. Using  this in above equation, we get
	\begin{align*}
		\begin{aligned}
			\int\limits_{\Sigma\setminus{G}}e^{-2\vp(t,x)}\lvert\PD_{\nu}u(t,x)\rvert^{2}\D S_{x}\D t\leq \frac{1}{\lambda\ve}\int\limits_{\Sigma_{+,\ve,\om}}\PD_{\nu}\vp e^{-2\vp(t,x)}\lvert\PD_{\nu}u(t,x)\rvert^{2}\D S_{x}\D t, \ 
		\end{aligned}
	\end{align*} 
	for $\omega\in\Sb^{n-1} \mbox{near}\ \omega_{0}\in\Sb^{n-1}$.
	Now using the Carleman estimate \eqref{bdryesti} and Equation \eqref{Equation for u linear}, we get 
	\begin{align}
		\begin{aligned}\label{Estimate over the boundary terms}
			&\Big\lvert\int\limits_{\Sigma\setminus  G}\partial_{\nu}u(t,x)\overline{v(t,x)}\D S_{x}\D t\Big\rvert^{2} \leq  C \lambda^{-1} \int\limits_{Q} e^{-2\vp(t,x)}\lvert\lb 2{A(t,x)}\cdot\nabla u_{2}(t,x)+\wt{q}(t,x)u_{2}(t,x)\rb\rvert^{2}\D x\D t.
		\end{aligned}
	\end{align}
	Using expression for $u_{2}$ from \eqref{growingsolu} and Equation \eqref{estimofrg}, we have
	\begin{align*}
		\int\limits_{Q} e^{-2\vp(t,x)}\lvert\lb 2{A(t,x)}\cdot\nabla u_{2}(t,x)+\wt{q}(t,x)u_{2}(t,x)\rb\rvert^{2}\D x\D t \le C \lambda^2.
	\end{align*}
	Hence using this in \eqref{Estimate over the boundary terms}, we get
	\begin{align*}
		\begin{aligned}
			\Big\lvert\int\limits_{\Sigma\setminus  G}\partial_{\nu}u(t,x)\overline{v(t,x)}\D S_{x}\D t\Big\rvert &\leq  C\lambda^{1/2},\ \text{for}\ \omega\in\Sb^{n-1}\ \text{such that} \ \lvert\om-\om_{0}\rvert\leq \epsilon.
		\end{aligned}
	\end{align*}
	This completes the proof of lemma.
\end{proof}
\section{Proof of theorem \ref{Main Theorem} and Corollary \ref{Corollary}}\label{proof of the main theorem}
In this section, we prove the uniqueness results.
Since from \eqref{Integral identity before using GO}, we have 
\begin{align*}
	\begin{aligned}
		\int\limits_{Q}\left(2{A(t,x)}\cdot\nabla_{x}u_{2}(t,x)+\wt{q}(t,x)u_{2}(t,x)\right)\overline{v(t,x)}\D x \D t=
		-\int\limits_{\Sigma\setminus{G}}\partial_{\nu}u(t,x)\overline{v(t,x)}\D S_{x}\D t.
		%+ \int\limits_{\Omega}u(T,x)\overline{v(T,x)}\D x.
	\end{aligned}
\end{align*} 
Now using Equation \eqref{Estimate for remainder equation}, we have
\begin{align*}
	\Big\lvert\int\limits_{Q}\left(2{A(t,x)}\cdot\nabla_{x}u_{2}(t,x)+\wt{q}(t,x)u_{2}(t,x)\right)\overline{v(t,x)}\D x \D t\Big\rvert\leq C\lambda^{1/2}.
\end{align*}
After dividing the above equation by $\lambda$ and taking $\lambda\to\infty$, we have 
\begin{align}\label{Integral idebtity after diving by lambda and taking lambda goes to infty}
	\lim_{\lambda\to\infty}\lb\frac{1}{\lambda}\int\limits_{Q}\left(2{A(t,x)}\cdot\nabla_{x}u_{2}(t,x)+\wt{q}(t,x)u_{2}(t,x)\right)\overline{v(t,x)}\D x \D t\rb=0.
\end{align}
Next using the expression for $u_{2}$ and $v$ from \eqref{growingsolu} and \eqref{decayingsolu} respectively, we have 
\begin{align*}
	\begin{aligned}
		\int\limits_{Q}\omega\cdot {A(t,x)}B_{g}(t,x)\overline{B_{d}(t,x)}  \D x\D t =0,\ \mbox{for all }\ \omega\in \Sb^{n-1} \ \mbox{such that} \ \lvert\omega-\omega_{0}\rvert\leq \epsilon.
	\end{aligned}
\end{align*}
This after using the expressions for $B_{g}(t,x)$ and $B_{d}(t,x)$ from Equations \eqref{expreofbg} and  \eqref{expreofbd} respectively, we get 
\begin{align}
	\begin{aligned}
		\int\limits_{Q}\omega\cdot {A(t,x)}\chi^{2}(t)e^{-i\xi\cdot x-i\tau t}\exp\lb-\int\limits_{0}^{\infty}\omega\cdot {A(t,x+s\omega)}\D s\rb \D x\D t =0,\ 
	\end{aligned}
\end{align}
$\mbox{for }\omega\in \Sb^{n-1} \ \mbox{with} \ \lvert\omega-\omega_{0}\rvert\leq \epsilon$.
Since the above identity holds for all $\chi\in C_{c}^{\infty}(0,T)$, therefore we get 
\begin{align}\label{final integral identity}
	\int\limits_{\mathbb{R}^{n}}\omega\cdot {A(t,x)}(x)e^{-i\xi\cdot x} \exp\lb-\int\limits_{0}^{\infty}\omega\cdot {A(t,x+s\omega)}ds\rb \D x  =0
\end{align}
where $\xi\cdot\omega=0$ for all $\omega$ with $\lvert\omega-\omega_{0}\rvert\leq \epsilon$. Now decompose $\mathbb{R}^{n}=\mathbb{R}\omega\oplus\omega^{\perp}$ and using this in the above equation, we have
\begin{align*}
	\begin{aligned}
		\int\limits_{\omega^{\perp}}e^{-i\xi\cdot k}\lb\int\limits_{\mathbb{R}}\omega\cdot {A(t,k+\tau\omega)}\exp\lb-\int\limits_{0}^{\infty}\omega\cdot {A\lb t,k+ \tau\omega +s\omega\rb}\D s\rb\D \tau \rb\D k=0,\  \mbox{for}\ \omega\ \mbox{with}\ \lvert\omega-\omega_{0}\rvert\leq \epsilon
	\end{aligned}
\end{align*}
here $\D k$ denotes the Lebesgue measure on $\omega^{\perp}$. After substituting $\tau+s=\tilde{s}$, we get
\begin{align}\label{Integral identity after using decomposition}
	\begin{aligned}
		\int\limits_{\omega^{\perp}}e^{-i\xi\cdot k}\lb \int\limits_{\mathbb{R}}\omega\cdot {A(t,k+\tau\omega)}\exp\lb-\int\limits_{\tau}^{\infty}\omega\cdot {A(t,k+\tilde{s}\omega)}\D \tilde{s}\rb\D \tau \rb\D k=0,\  \text{for } \omega \ \text{with}\ \lvert\omega-\omega_{0}\rvert\leq \epsilon.
	\end{aligned}
\end{align}
Now 
\begin{align*}
	\begin{aligned}
		& \int\limits_{\omega^{\perp}}e^{-i\xi\cdot k}\lb \int\limits_{\mathbb{R}}\omega\cdot {A(t,k+\tau\omega)}\exp\lb-\int\limits_{\tau}^{\infty}\omega\cdot {A(t,k+\tilde{s}\omega)}\D \tilde{s}\rb\D \tau \rb\D k\\
		&\  =\int\limits_{\omega^{\perp}}e^{-i\xi\cdot k}\int\limits_{\mathbb{R}}\frac{\PD}{\PD\tau}\exp\lb -\int\limits_{\tau}^{\infty}\omega\cdot {A(t,k+s\omega)}\D s\rb\D \tau\D k\\
		&\  =\int\limits_{\omega^{\perp}}e^{-i\xi\cdot k} \lb 1-\exp\lb-\int\limits_{\mathbb{R}} \omega\cdot {A(t,k+s\omega)}\D s\rb\rb\D k.
	\end{aligned}   
\end{align*}
Combining this with \eqref{Integral identity after using decomposition}, we get 
\begin{align*}
	\int\limits_{\mathbb{R}}\omega\cdot {A(t,k+s\omega)}\D s=0, \ \text{for} \ k\in\omega^{\perp} \  \text{with}\ \ \lvert\omega-\omega_{0}\rvert\leq \epsilon.
\end{align*}
Now using the decomposition $\mathbb{R}^{n}=\mathbb{R}\omega\oplus\omega^{\perp}$ in the above equation, we get 
\begin{align}\label{Ray transform}
	\int\limits_{\mathbb{R}}\omega\cdot {A(t,x+s\omega)}\D s=0, \ \text{for} \ x\in \Rb^{n} \  \text{with}\ \ \lvert\omega-\omega_{0}\rvert\leq \epsilon.
\end{align}
Thus we have the ray transform of vector field $A$ is vanishing  in a very small enough  neighbourhood of  fixed direction $\omega_{0}$. In order to get the uniqueness for vector field term $A$, we need to invert this ray transform which we will do in the following lemma:

\begin{lemma} \label{Ray transform lemma}
	Let ${n\geq 2}$ and $F=(F_{1},F_{2},\cdots,F_{n})$ be a real-valued time-dependent vector field with $F_{j}\in C_{c}^{\infty}(Q)$ for all $1\leq j\leq n$. Suppose for each $t\in (0,T)$ we have
	\begin{align*}
		IF(t,x,\omega):=\int\limits_{\mathbb{R}}\omega\cdot F(t,x+s\omega)\D s=0
	\end{align*}
	for all $\omega\in \Sb^{n-1}$ with  $\lvert\omega-\omega_{0}\rvert\leq \epsilon$, for some $\epsilon>0$ and for all $x\in\mathbb{R}^{n}$. Then for each  $t\in (0,T)$ there exists a $\Phi(t,\cdot)\in C_{c}^{\infty}(\Omega)$ such that $F(t,x)=\nabla_{x}\Phi(t,x)$.
	\begin{proof} The proof uses the  arguments similar to the one used in  \cite{Krishnan_Vashisth_Relativistic,Siamak_Light_Ray_Vector_Field,Stefanov_light_Ray} for the case of light ray transforms. We assume that $t\in(0,T)$ is arbitrary but fixed.
		We have the ray transform of $F$ at $x\in\Rb^{n}$ in the direction of $\omega\in\Sb^{n-1}$ is given by 
		\begin{align*}
			\begin{aligned}
				IF(t,x,\omega)=\int\limits_{\mathbb{R}}\omega\cdot F(t,x+s\omega)\D s.
			\end{aligned}
		\end{align*}
		Now let  $\eta:=(\eta_{1},\eta_{2},\cdots,\eta_{n})\in\Rb^{n}$ be arbitrary and denote $\omega:=(\omega^{1},\omega^{2},\cdots,\omega^{n})\in \Sb^{n-1}$. Then we have 
		\begin{align}\label{I equation Ray transform}
			\begin{aligned}
				(\eta\cdot\nabla_{x})IF(t,x,\omega)=\sum_{i,j=1}^{n}\int\limits_{\Rb}\omega^{i}\eta_{j}\PD_{j}F_{i}(t,x+s\omega)\D s.
			\end{aligned}
		\end{align}
		Since $F$ has compact support  therefore using the Fundamental theorem of calculus, we have 
		\begin{align*}
			\int\limits_{\Rb}\frac{\D}{\D s}(\eta\cdot F)(t,x+s\omega)\D s=0
		\end{align*}
		which gives
		\begin{align}\label{II equation Ray transform}
			\sum_{i,j=1}^{n}\int\limits_{\Rb}  \omega^{i}\eta_{j}\PD_{i}F_{j}(t,x+s\omega)\D s=0.
		\end{align}
		Subtracting \eqref{II equation Ray transform} from \eqref{I equation Ray transform}, we get 
		\begin{align}\label{Vanishing of transvese ray transform}
			\sum_{i,j=1}^{n}\int\limits_{\Rb}\omega^{i}\eta_{j}h_{ij}(t,x+s\omega)\D s=0,  \ \text{for} \  x\in\Rb^{n} \ \text{and}\  \omega\in \Sb^{n-1}\ \text{near a fixed }\ \omega_{0}\in\Sb^{n-1}
		\end{align}
		where $h_{ij}$ is an $n\times n$ matrix with entries 
		\begin{align*}
			h_{ij}(t,x)=\lb\PD_{j}F_{i}-\PD_{i}F_{j}\rb(t,x),\ \text{for} \ 1\leq i,j\leq n.
		\end{align*}
		Define the Fourier transform of $\sum_{i,j=1}^{n}\omega^{i}\eta_{j}{h}_{ij}(t,x)$ with respect to space variable $x$ by 
		\begin{align*}
			\sum_{i,j=1}^{n}\omega^{i}\eta_{j}\widehat{h}_{ij}(t,\xi)= \sum_{i,j=1}^{n}\int\limits_{\mathbb{R}^{n}}\omega^{i}\eta_{j}h_{ij}(t,x)e^{-i\xi\cdot x}\D x,\ \xi\in\Rb^{n}.
		\end{align*}
		Now decomposing $\mathbb{R}^{n}=\mathbb{R}\omega\oplus\omega^{\perp}$ and using \eqref{Vanishing of transvese ray transform}, we get
		\begin{align}\label{Vanishing of FT of omegai eta j hij}
			\begin{aligned}
				\sum_{i,j=1}^{n}\omega^{i}\eta_{j}\widehat{h}_{ij}(t,\xi)=0,\ \text{for all}\ \eta\in\mathbb{R}^{n},\ \xi\in\omega^{\perp}\ \text{and}\ \omega\ \text{near}\ \omega_{0}.  
			\end{aligned}
		\end{align}
		The goal is to prove that $\widehat{h}_{ij}(t,\xi)=0, \  \mbox{for}\  \xi\in\omega^{\perp}\ \text{with}\ \omega\ \text{near}\ \omega_{0}$ \mbox{and  for each}  $ t\in (0,T) $.		
		From the definition of $ \widehat{h}_{ij}(t,\xi)$, it is clear that\[ \widehat{h}_{ii}(t,\xi)=0\  \mbox{and}\ \widehat{h}_{ij}(t,\xi)=-\widehat{h}_{ji}(t,\xi),\ \mbox{for}\  1\le i,j\le n.\] 
		For $n=2$,  equation \ref{Vanishing of FT of omegai eta j hij} gives us
		\begin{align}\label{Vanishin FT of h12 hat}
			(\omega^1 \eta_2- \omega^2\eta_1)\widehat{h}_{12}(t,\xi)=0, \   \text{for}\ \eta\in\mathbb{R}^{2},\ \xi\in\omega^{\perp}\ \text{and}\ \omega\ \text{near}\ \omega_{0}.
		\end{align} 		
		Now choosing $ \eta = \lb \omega_{2},-\omega_{1}\rb\in \omega^{\perp}$ in \eqref{Vanishin FT of h12 hat}, we get $ \widehat{h}_{12}(\xi)=0. $ Next we show that $\widehat{h}_{ij}(t,\xi)=0$ when $n\geq 3$. 
		Let $\{e_{j}: 1\leq j\leq n\}$ be the standard basis for $\mathbb{R}^{n}$ where $e_{j}$ is is given by
		\[e_{j}:=(0,0,\cdots,0,{\underbrace{1}_{j\text{th}}},0,\cdots,0)\]
		and 
		for simplicity we fix $\omega_{0}=e_{1}$. Now let $\xi_{0}=e_{2}$ 
		be a fixed vector in $\mathbb{R}^{n}$. Our first aim is to show that $\widehat{h}_{ij}(\xi_{0})=0$, for all $1\leq i,j\leq n$, then later we will prove that $\widehat{h}_{ij}(t,\xi)=0$ for $1\leq i,j\leq n$ and  $\xi$ near $\xi_{0}$. Following \cite{Krishnan_Vashisth_Relativistic}, consider a small perturbation $\omega_{0}(a)$ of vector $\omega_{0}=e_{1}$ by
		\[\omega_{0}(a):= \cos a e_{1}+\sin a e_{k}\ \text{where} \ 3\leq k\leq n.\]
		Then we have $\omega_{0}(a)$ is near $\omega_{0}$ for $a$ near $0$ and $\xi_{0}\cdot\omega_{0}(a)=0$. Hence using these choices of $\omega_{0}(a)$ and $\eta=e_{j}$ in\eqref{Vanishing of FT of omegai eta j hij}, we have 
		\begin{align*}
			\begin{aligned}
				\cos a\widehat{h}_{1j}(t,\xi_{0})+\sin a\widehat{h}_{kj}(t,\xi_{0})=0,\ \text{for} \ 1\leq j\leq n, \  \ 3\leq k\leq n \ \text{and}\  a\  \text{near}\  0.
			\end{aligned}
		\end{align*}
		This gives us 
		\begin{align*}
			\widehat{h}_{1j}(t,\xi_{0})=0,\ \widehat{h}_{kj}(t,\xi_{0})=0,\ \text{for} \ 1\leq j\leq n, \text{and} \ 3\leq k\leq n.
		\end{align*}
		After using the fact that $\widehat{h}_{ij}=-\widehat{h}_{ji}$ for $1\leq i,j\leq n$, we get 
		\begin{align*}
			\widehat{h}_{ij}(t,\xi_{0})=0,\ \text{for} \ 1\leq i,j\leq n.
		\end{align*}
		Next we show that $\widehat{h}_{ij}(t,\xi)=0$ for $\xi\in\omega^{\perp}$ with $\omega$ near $\omega_{0}$. 
		Using the spherical  co-ordinates, we choose  $ \xi \in\Sb^{n-1} $ as follows
		\begin{align*}
			\xi^{1}& = \sin\phi_1 \cos\phi_2\\
			\xi^{2}&= \cos\phi_1\\
			\xi^{3}&= \sin\phi_1 \sin\phi_2 \cos\phi_3\\
			\vdots\\
			\xi^{n-1}&=  \sin\phi_1 \sin\phi_2 \cdots \sin\phi_{n-2}\cos\theta\\
			\xi^{n}&=  \sin\phi_1 \sin\phi_2 \cdots \sin\phi_{n-2}\sin\theta.
		\end{align*}
		Let $ A $ be an orthogonal matrix such that $ A \xi=e_{2} $, where $ A $ is given by
		\begin{align}\label{orthogonal matrix}
			A=\begin{pmatrix}
				\frac{\PD \xi^{1}}{\PD \phi_1}&\frac{\PD \xi^{2}}{\PD \phi_1}& \frac{\PD \xi^{3}}{\PD \phi_1}&\cdots& \frac{\PD \xi^{n}}{\PD \phi_1}\\
				\xi^1 &\xi^2& \xi^3&\cdots&\xi^n\\
				a_{31}& a_{32}&a_{33}&\cdots&a_{3n}\\
				\vdots& \vdots&\vdots& \ddots&\vdots\\
				a_{k1}& a_{k2}&a_{k3}&\cdots&a_{kn}\\
				a_{k+11}& a_{k+11}&a_{k+13}&\cdots&a_{k+1n}\\
				\vdots& \vdots&\vdots& \ddots&\vdots\\
				a_{n1}& a_{n2}&a_{n3}&\cdots&a_{nn}	
			\end{pmatrix}
			.
		\end{align}
		Now	choose
		\[\widetilde{\omega}= \lb\frac{\PD \xi^{1}}{\PD \phi_1},\frac{\PD \xi^{2}}{\PD \phi_1}, \frac{\PD \xi^{3}}{\PD \phi_1},\cdots, \frac{\PD \xi^{n}}{\PD \phi_1}\rb \in \mathbb{S}^{n-1} \] 
		then $ \widetilde{\omega} $ is near $ \omega_{0} =e_1$ when $ \phi_{i} $'s
		and $ \theta $ are close to $ 0 $. Next choose $ \omega_{0}(a) = \cos a e_1 +\sin a e_{l} $ with $l\neq 2$, then $\omega_{0}(a)$ is close to $ e_1 $
		when $ a$ is close to zero.  Now define $\omega(a)$ by 
		\begin{align*}
			\omega(a):= A^{T} \omega_{0}(a)=
			\begin{pmatrix}
				\cos a   \cos \phi_1 \cos\phi_2+ a_{l1} \sin a \\
				- \cos a  \sin \phi_1+a_{l2}  \sin a \\
				\vdots \\
				\cos a  \cos \phi_1\sin\phi_2 \cdots\sin\phi_{n-2}\cos \theta + a_{ln} \sin a 
			\end{pmatrix}
			.
		\end{align*}	
		Then we have $\omega(a)$ is close $\widetilde{\omega}$ for $a$ near $0$ and $\tilde{\omega}$ is close to $\omega_{0}$ when $\phi_{i}$ and $\theta$ are close to zero. Also we can see that  $\omega(a)\cdot\xi=A^{T}\omega_{0}(a)\cdot\xi=\omega_{0}(a)\cdot A\xi=\omega_{0}(a)\cdot e_{2}=0$, hold  because of the choice of $\omega_{0}(a)$.
		Hence using these choices of $\omega(a)$ and choosing  $ \eta = e_{j} $ in \eqref{Vanishing of FT of omegai eta j hij}, we have
		\begin{align*}
			\cos a  \left( \sum_{i=1}^{n} \frac{\PD \xi^{i}}{\PD \phi_{1}} \widehat{h}_{ij}(t,\xi) \right) +\sin a 
			\left(\sum_{i=1}^{n} a_{li} \widehat{h}_{ij}(t,\xi)  \right)=0,\ \mbox{for} 1\le j \le n \  \mbox{and}\  a \ \mbox{near}\  0.
		\end{align*} 
		Now since $ \sin a $ and $ \cos a $ are linearly independent, therefore we get
		\begin{equation}\label{independent relations}
			\begin{aligned}
				\sum_{i=1}^{n} \frac{\PD \xi^{i}}{\PD \phi_{1}} \widehat{h}_{ij}(t,\xi)&=0,\ \mbox{for all}\ 1\leq j\leq n\\
				\sum_{i=1}^{n} a_{li} \widehat{h}_{ij}(t,\xi)&=0, \ \mbox{for all}\ 1\leq j\leq n\ \mbox{and}\  l\neq 2.
			\end{aligned} 
		\end{equation}
		Above equations can be written as 
		\begin{align}\label{Syetem of equations}
			\begin{pmatrix}
				\frac{\PD \xi^{1}}{\PD \phi_1}&\frac{\PD \xi^{2}}{\PD \phi_1}& \frac{\PD \xi^{3}}{\PD \phi_1}&\cdots& \frac{\PD \xi^{n}}{\PD \phi_1}\\
				a_{31}& a_{32}&a_{33}&\cdots&a_{3n}\\
				\vdots& \vdots&\vdots& \ddots&\vdots\\
				a_{k1}& a_{k2}&a_{k3}&\cdots&a_{kn}\\
				a_{k+1,1}& a_{k+1,2}&a_{k+1,3}&\cdots&a_{k+1,n}\\
				\vdots& \vdots&\vdots& \ddots&\vdots\\
				a_{n1}& a_{n2}&a_{n3}&\cdots&a_{nn}.
			\end{pmatrix} \begin{pmatrix}
				\widehat{h}_{1j}(t,\xi)\\
				\widehat{h}_{2j}(t,\xi)\\
				\widehat{h}_{3j}(t,\xi)\\
				\vdots\\
				\vdots\\
				\vdots\\
				\widehat{h}_{nj}(t,\xi)
			\end{pmatrix} =0.
		\end{align}
		Now let us define matrix $ B $ and a n-vector $\textbf{h}_{j}$ as follows:
		\begin{align*}
			B= \begin{pmatrix}
				\frac{\PD \xi^{1}}{\PD \phi_1}&\frac{\PD \xi^{2}}{\PD \phi_1}& \frac{\PD \xi^{3}}{\PD \phi_1}&\cdots& \frac{\PD \xi^{n}}{\PD \phi_1}\\
				a_{31}& a_{32}&a_{33}&\cdots&a_{3n}\\
				\vdots& \vdots&\vdots& \ddots&\vdots\\
				a_{k1}& a_{k2}&a_{k3}&\cdots&a_{kn}\\
				a_{k+1,1}& a_{k+1,2}&a_{k+1,3}&\cdots&a_{k+1,n}\\
				\vdots& \vdots&\vdots& \ddots&\vdots\\
				a_{n1}& a_{n2}&a_{n3}&\cdots&a_{nn}
			\end{pmatrix}
			and \ 
			\textbf{h}_{j}(t,\xi)=
			\begin{pmatrix}
				\widehat{h}_{1j}(t,\xi)\\
				\widehat{h}_{2j}(t,\xi)\\
				\widehat{h}_{3j}(t,\xi)\\
				\vdots\\
				\vdots\\
				\vdots\\
				\widehat{h}_{nj}(t,\xi)
			\end{pmatrix}.
		\end{align*}
		Using these, we have equation \eqref{Syetem of equations}, can be written as
		\begin{align}\label{System of equations in compact form}
			B\textbf{h}_{j}(t,\xi)=0,\ \text{for} \ 1\leq j\leq n.
		\end{align}
		Note that the matrix $ B $ is obtained from $ A $ by removing the second row and it is $(n-1)\times n$ matrix. From the definition of $A$ it is clear that rank of $ A $ is $n$, so the rank of $ B $ is $ n-1 $. $ i.e.$ there exists at-least one non-zero minor of order $n-1$ of the matrix $B$. Without loss of generality assume $ B' $ is non-zero minor of order $ n-1 $,  where $B'$ is given by 
		\begin{align*}
			B'=\begin{pmatrix}
				\frac{\PD \xi^{2}}{\PD \phi_1}& \frac{\PD \xi^{3}}{\PD \phi_1}&\cdots& \frac{\PD \xi^{n}}{\PD \phi_1}\\
				a_{32}&a_{33}&\cdots&a_{3n}\\
				\vdots&\vdots& \ddots&\vdots\\
				a_{k2}&a_{k3}&\cdots&a_{kn}\\
				a_{k+1,2}&a_{k+1,3}&\cdots&a_{k+1,n}\\
				\vdots&\vdots& \ddots&\vdots\\
				a_{n2}&a_{n3}&\cdots&a_{nn}
			\end{pmatrix}.
		\end{align*}
		Now using the fact $ \widehat{h}_{11}(t,\xi)=0 $ in \eqref{System of equations in compact form}, we have
		\begin{align}\label{equations for B'}
			B' \textbf{h}'_{1}(t,\xi)=0
		\end{align}
		where $ \textbf{h}'_{1}=\left(\widehat{h}_{21}(t,\xi),
		\widehat{h}_{31}(t,\xi),\cdots,
		\widehat{h}_{n1}(t,\xi)\right)^T.$  Since $ B' $ has full rank therefore $  \textbf{h}'_{1}(t,\xi)=0 $. Also using the fact $ \widehat{h}_{ij}(t,\xi)= -\widehat{h}_{ji}(t,\xi) $ and $ \textbf{h}_{1}(t,\xi)=0 $,  we have 
		\begin{align}\label{sysetm of equ for B'}
			B'\textbf{h}'_{j}(t,\xi)=0, \  \mbox{for} \  2\le j\le n
		\end{align}
		where $\textbf{h}_{j}'$ is an $(n-1)$ vector obtained after deleting $jth$ entry from $\textbf{h}_{j}$. %=(\widehat{h}_{1j},\widehat{h}_{2j},\cdots,\widehat{h}_{nj})^{T}$.
		Now using \eqref{equations for B'}  and \eqref{sysetm of equ for B'} in \eqref{System of equations in compact form}, we get	
		\[ \textbf{h}_{j}(t,\xi)=0,\ \ \mbox{for} \ 1\le j \le n\]
		which gives us 
		\[\widehat{h}_{ij}(t,\xi) =0,\ \mbox{for}\ 1\leq i,j\leq n\ \mbox{and\ $\xi$ near $e_{2}$.} \]
		Since $\widehat{h}_{ij}$ are compactly supported therefore using the  Paley-Wiener theorem, we have 
		\[\widehat{h}_{ij} (t,\xi) = 0,\ \mbox{for} \ 1\leq i, j \leq n\ \xi\in \Rb^{n} \ \mbox{and} \ t\in (0,T).\] 
		Fourier inversion formula gives us
		$h_{ij} (t,x) = 0$ for  $x\in\mathbb{R}^{n}$ and for each $t\in (0,T)$. 
		Finally after using the definition of $h_{ij}(t,x)$ and the Poincar\'e lemma, there exists a $\Phi(t,\cdot)\in C_{c}^{\infty}(\mathbb{R}^{n})$ such that $F(t,x)=\nabla_{x}\Phi(t,x)$ for $x\in\Omega$ and  for each $t\in (0,T)$.  This completes the proof of Lemma \ref{Ray transform lemma}.
	\end{proof}

\end{lemma}

\subsection{Proof of Theorem \ref{Main Theorem}} \label{Proof for Main theorem}       Using \eqref{Ray transform} and the fact that $A$ is time-independent, we have
\[	\int\limits_{\mathbb{R}}\omega\cdot {A(x+s\omega)}\D s=0, \ \text{for} \ x\in \Rb^{n} \  \text{with}\ \ \lvert\omega-\omega_{0}\rvert\leq \epsilon.\]  
Hence using Lemma \ref{Ray transform lemma} in the above equation,  there exists $\Phi\in W^{2,\infty}_{0}(\Omega)$ such that \
\begin{align}\label{Uniqueness for A}
	A(x)=\nabla_{x}\Phi(x),\ x\in\Omega.
\end{align}
This completes the proof for recovery of  convection term $A(x)$.  Next we prove the uniqueness for the density coefficients  $q_{i}(t,x)$ for $i=1,2$. 	Since from  \eqref{Uniqueness for A}, we have $A^{(2)}(x)-A^{(1)}(x)=\nabla_{x}\Phi(x)$ for some $\Phi\in W^{2,\infty}_{0}(\Omega)$.  {Now if replace the pair $(A^{(1)},q_{1})$ by $(A^{(3)},q_{3})$ where $A^{(3)}=A^{(1)}+\nabla_{x}\Phi$ and $q_{3}=q_{1}$ then using the fact that $\Phi\in W^{2,\infty}_{0}(\Omega)$  and Equation  \eqref{Equality of DN map}, we get $\Lambda_{A^{(3)},q_{3}}=\Lambda_{A^{(2)},q_{2}}$. Now repeating the previous arguments and Lemma \ref{Ray transform lemma}, there exists $\Phi_{1}\in W^{2,\infty}_{0}(\Omega)$ such that 
	\[A^{(3)}(x)-A^{(2)}(x)=\nabla_{x}\Phi_{1}(x)\]
	which gives us $A^{(3)}(x)=A^{(2)}(x)$ for $x\in\Omega$. Hence using pairs $(A^{(3)},q_{3})$ and $(A^{(2)},q_{2})$ in \eqref{Integral identity before using GO} and the fact that $q_{3}=q_{1}$, we get 
	\begin{align*}
		\int\limits_{Q}q(t,x)u_{2}(t,x)\overline{v(t,x)}\D x \D t =
		-\int\limits_{\Sigma\setminus{G}}\partial_{\nu}u(t,x)\overline{v(t,x)}\D S_{x}\D t 
	\end{align*}
	where $q(t,x):= q_{1}(t,x)-q_{2}(t,x)$.	Now using the expressions for $u_{2}$ and $v$ from \eqref{growingsolu} and \eqref{decayingsolu}  respectively and taking $\lambda\to\infty$, we get 
	\begin{align*}
		\int\limits_{Q}q(t,x)e^{-i(\tau t+x\cdot\xi)}\D x\D t=0,\ \text{for} \ \tau\in\mathbb{R} \ \mbox{and} \ \xi\in\omega^{\perp},  \ \mbox{where} \ \omega\in \Sb^{n-1} \ \mbox{is near} \ \omega_{0}.
	\end{align*}
	Since $q\in L^{\infty}(Q)$ is zero outside $Q$ therefore by using the Paley-Wiener theorem we have
	$q_{1}(t,x)=q_{2}(t,x)$ for $(t,x)\in Q.$ This completes the proof of Theorem \ref{Main Theorem}.
	\subsection{Proof of Corollary \ref{Corollary}}\label{Proof of corollary}{Using Equation \eqref{Ray transform} and Lemma \eqref{Ray transform lemma}, we have  for every $t\in (0,T)$ there exists $\Phi(t,\cdot)\in W^{2,\infty}_{0}(\Omega)$ such that 
		\begin{align}\label{Uniqueness for A time-dependent}
			A^{(2)}(t,x)-A^{(1)}(t,x)=\nabla_{x}\Phi(t,x),\ (t,x)\in Q.
		\end{align}
		Now using Equations \eqref{Divergence free} and \eqref{Uniqueness for A time-dependent}, we have 
		\begin{align*}
			\begin{aligned}
				\begin{cases}
					&\Delta_{x}\Phi(t,x)=0,\ x\in \Omega \ \mbox{and for each  $t\in (0,T)$}\\
					&\Phi(t,x)=0,\ x\in\PD\Omega \ \mbox{and for each  $t\in (0,T)$}.
				\end{cases}
			\end{aligned}
		\end{align*}
		Using the unique solvability for the above equation, we have $\Phi(t,x)=0$ for $(t,x)\in Q$.  Thus from Equation \eqref{Uniqueness for A time-dependent}, we get $A^{(2)}(t,x)=A^{(1)}(t,x)$ for $(t,x)\in Q$. Using this in \eqref{Integral identity before using GO} and repeating the previous arguments, we get $q_{1}(t,x)=q_{2}(t,x)$, $(t,x)\in Q$. 
	
	\section*{Acknowledgments}	
		We thank the anonymous referees for	useful comments which helped us to improve the paper.
	MV thanks Ibtissem Ben A\"icha and  Gen Nakamura for the discussions on this problem.  Both the authors are thankful to Venky Krishnan for stimulating discussions and many useful suggestions which helped us to improve the paper. SS was partially supported by Matrics grant MTR/2017/000837.  The work of MV was supported by NSAF grant (No. U1930402).

\end{document}